\documentclass[12pt]{article}
\usepackage{booktabs}
\usepackage{caption}
\usepackage{mathrsfs}
\usepackage{amsmath}
\usepackage{amsfonts,amsthm,amssymb,mathrsfs,bbding}
\usepackage{graphics,multicol}
\usepackage{graphicx}
\usepackage{color}
\usepackage{enumerate}
\usepackage{caption}
\usepackage{rotating}
\usepackage{lscape}
\usepackage{longtable}
\allowdisplaybreaks[4]
\usepackage{tabularx}
\usepackage[
colorlinks,
anchorcolor=blue,
linkcolor=blue,
citecolor=blue]{hyperref}
\usepackage{extarrows}
\usepackage{cite}
\usepackage{latexsym,bm}
\usepackage{mathtools}
\evensidemargin=\oddsidemargin\topmargin=-1.5cm

\newtheorem{thm}{Theorem}[section]
\newtheorem{prob}{Problem}[section]
\newtheorem{lem}[thm]{Lemma}
\newtheorem{cor}[thm]{Corollary}

\newtheorem{remark}{Remark}[section]
\newtheorem{exam}{Example}
\newtheorem{conj}{Conjecture}[section]
\newtheorem{claim}{Claim}[section]
\theoremstyle{definition}

\addtocounter{section}{0}
\begin{document}
\title{The second eigenvalue of some normal Cayley graphs of high transitive groups\footnote{The second author is supported by NSFC  11531011 and 11671344, and the third author is supported by the NSF grants DMS-1600768 and CIF-1815922.}}
\author{{\small Xueyi Huang$^1$, \ \ Qiongxiang Huang$^2$\footnote{
Corresponding author.}\setcounter{footnote}{-1}\footnote{
\emph{E-mail address:} huangxymath@gmail.com (X. Huang), huangqx@xju.edu.cn (Q. Huang), cioaba@udel.edu (S.M. Cioab\u{a}).}}\ \ \small and  Sebastian M. Cioab\u{a}{$^3$}\\[2mm]
\scriptsize
$^1$School of Mathematics and Statistics, Zhengzhou University, Zhengzhou, Henan 450001, P. R. China\\
\scriptsize
$^2$College of Mathematics and Systems Science, Xinjiang University, Urumqi, Xinjiang 830046, P. R. China\\
\scriptsize
$^3$Department of Mathematical Sciences,  University of Delaware, 501 Ewing Hall, Newark, DE 19716, USA}

\date{}
\maketitle
{\flushleft\large\bf Abstract}
Let $\Gamma$ be a finite group acting transitively on $[n]=\{1,2,\ldots,n\}$, and  let $G=\mathrm{Cay}(\Gamma,T)$ be a Cayley graph of $\Gamma$. The graph $G$ is called  normal if $T$ is closed under conjugation. In this paper, we obtain an upper bound for the second (largest) eigenvalue of the adjacency matrix of the graph $G$ in terms of the second eigenvalues of certain subgraphs of $G$ (see Theorem 2.6). Using this result, we develop a recursive method to  determine the second eigenvalues of certain  Cayley graphs of $S_n$ and we determine the second eigenvalues  of a majority of the connected normal Cayley graphs (and some of their subgraphs) of $S_n$  with  $\max_{\tau\in T}|\mathrm{supp}(\tau)|\leq 5$, where $\mathrm{supp}(\tau)$ is the set of points in $[n]$ non-fixed by $\tau$. 

\begin{flushleft}
\textbf{Keywords:}  The second eigenvalue;  Normal Cayley graph; Symmetric group.
\end{flushleft}
\textbf{AMS Classification:} 05C50

\section{Introduction}\label{s-1}
Let $G=(V(G),E(G))$ be a simple undirected graph of order $n$ with adjacency matrix $A(G)$.   The eigenvalues of $A(G)$, denoted by $\lambda_1(G)\geq \lambda_2(G)\geq \cdots\geq \lambda_n(G)$, are also called the  \emph{eigenvalues} of $G$.  For a $k$-regular graph $G$, the spectral gap  $\lambda_1(G)-\lambda_2(G)=k-\lambda_2(G)$ is closely related to the connectivity and expansion properties of $G$ \cite{Alon,Alon1,Dodziuk,Fiedler,Lubotzky1,Mohar,Hoory}.

Let $\Gamma$ be  a finite group, and let $T$ be a subset of  $\Gamma$ such that $e\not\in T$ ($e$ is the identity element of $\Gamma$) and  $T=T^{-1}$. The \emph{Cayley graph} $\mathrm{Cay}(\Gamma,T)$ of $\Gamma$ with respect to  $T$ (called \emph{connection set})  is defined as the undirected graph with vertex set $\Gamma$ and  edge set $\{\{\gamma,\tau\gamma \}\mid\gamma\in\Gamma,\tau\in T\}$. Clearly,  $\mathrm{Cay}(\Gamma,T)$ is a regular graph which is connected if and only if $T$ is a generating  subset of  $\Gamma$. A Cayley graph $\mathrm{Cay}(\Gamma,T)$ is called \emph{normal} if $T$ is closed under conjugation.

Let $S_n$ be the symmetric group on $[n]=\{1,2,\ldots,n\}$ with $n\geq 3$, and  $T$ a subset of $S_n$ consisting of transpositions. The \emph{transposition graph} $\mathrm{Tra}(T)$ of $T$  is defined as the graph with vertex set $\{1,2,\ldots,n\}$ and with an edge connecting two vertices $i$ and $j$ if and only if $(i,j)\in T$. It is known that $T$ can generate $S_n$ if and only if  $\mathrm{Tra}(T)$ is connected \cite{Godsil1}. In 1992,  Aldous \cite{Aldous} (see also \cite{Friedman,Cesi}) conjectured that the spectral gap of $\mathrm{Cay}(S_n,T)$ is equal to  the algebraic connectivity (second least Laplacian eigenvalue) of $\mathrm{Tra}(T)$.  Earlier efforts of several researchers solved various special cases of Aldous' conjecture. For instance, Diaconis and Shahshahani \cite{Diaconis}, and Flatto, Odlyzko and Wales \cite{Flatto}   confirmed the conjecture for  $\mathrm{Tra}(T)$ being a complete graph and a star, respectively;   Handjani and Jungreis \cite{Handjani} confirmed the conjecture for $\mathrm{Tra}(T)$ being a tree; Friedman \cite{Friedman} proved that if $\mathrm{Tra}(T)$ is a bipartite graph then the spectral gap of $\mathrm{Cay}(S_n,T)$ is at most the algebraic connectivity of $\mathrm{Tra}(T)$; Cesi \cite{Cesi} confirmed the conjecture for $\mathrm{Tra}(T)$ being a complete multipartite graph.   At last, Caputo, Liggett and Richthammer \cite{Caputo} completely confirmed the conjecture in 2010, their proof is an ingenious combination of two ingredients: a nonlinear mapping in the group algebra 􏱆 􏰾􏱓$\mathbb{C}S_n$ which permits a proof by induction on $n$, and a quite complicated estimate named the octopus inequality  (see also \cite{Cesi2} for a self-contained algebraic proof).  Very recently, Cesi \cite{Cesi3} proved an analogous result of Aldous' conjecture (now theorem) for the Weyl group $W(B_n)$. Most of the above results rely heavily on the representation theory of the symmetric group $S_n$.

The second eigenvalues of Cayley graphs of the symmetric group $S_n$ or the alternating groups $A_n$ have been  determined also for some  special generators that  are not transpositions. For $1\leq i<j\leq n$, let $r_{i,j}\in S_n$ be defined as
$$r_{i,j}=\left(\begin{matrix}
1~\cdots~ i-1&i&i+1~ \cdots~ j-1&j&j+1~\cdots~ n\\
1~\cdots~ i-1&j&j-1 ~\cdots~ i+1&i&j+1~\cdots~ n\\
\end{matrix}\right).$$
In \cite{Cesi1}, Cesi  proved that the second eigenvalue of the pancake graph $\mathcal{P}_n=\mathrm{Cay}(S_n,$ $\{r_{1,j}\mid 2\leq j\leq n\})$ is equal to $n-2$. In \cite{Chung},  Chung and Tobin determined the second eigenvalues of the reversal graph $R_n=\mathrm{Cay}(S_n,\{r_{i,j}\mid 1\leq i<j\leq n\})$ and a family of graphs that generalize the pancake graph $\mathcal{P}_n$. In \cite{Parzanchevski}, Parzanchevski and Puder  proved that, for large enough $n$, if $S\subseteq S_n$ is a full conjugacy class generating $S_n$ then the second eigenvalue of $\mathrm{Cay}(S_n,S)$  is always associated with one of  eight low-dimensional representations of $S_n$. In \cite{Huang1}, the  authors  determined the second eigenvalues of the alternating group graph $AG_n=\mathrm{Cay}(A_n,\{(1,2,i),(1,i,2)\mid 3\leq i\leq n\})$ (introduced by Jwo, Lakshmivarahan and Dhall \cite{Jwo}), the extended alternating group graph $EAG_n=\mathrm{Cay}(A_n,\{(1,i,j),(1,j,i)\mid 2\leq i<j\leq n\})$ and the complete alternating group graph $CAG_n=\mathrm{Cay}(A_n,\{(i,j,k),(i,k,j)\mid 1\leq i<j<k\leq n\})$ (defined by Huang and Huang \cite{Huang}).

Suppose that $\Gamma$ is a finite group acting transitively  on $[n]$ and let $G=\mathrm{Cay}(\Gamma,T)$.  In the present paper, we first show that, for each $i\in [n]$, the left coset decomposition  of $\Gamma$ with respect to the stabilizer subgroup $\Gamma_i$  is  an equitable partition of $G$, and all these equitable partitions share the same quotient matrix $B_\Pi$. Based on this fact, we also prove that  those eigenvalues of $G$ not belonging to $B_\Pi$ can be bounded above by the sum of  second eigenvalues of some subgraphs of $G$. Now suppose further that $G$ is connected and normal, and that the action of $\Gamma$ on $[n]$ is of high transitivity. Using the previous result, we reduce the problem of proving  $\lambda_2(G)=\lambda_2(B_\Pi)$ to that of verifying the result for  some smaller graphs. This leads to a recursive procedure for determining the second eigenvalue of $G$. As applications, we determine the second eigenvalues of a majority of connected normal Cayley graphs  of $S_n$ with  $\max_{\tau\in T}|\mathrm{supp}(\tau)|\leq 5$ (see Theorem \ref{symmetric-thm} and Table \ref{tab-2}), where $\mathrm{supp}(\tau)$ is the set of points in $[n]$ non-fixed by $\tau$. There are $56$ families of such graphs, and we determine the second eigenvalues for $41$ families of them. In the process, we also determine the second eigenvalues of some subgraphs (over one hundred families) of these $41$ families of normal Cayley graphs. From these results we can determine the spectral gap of  $\mathrm{Cay}(S_n,\{(p,q)\mid 1\leq p,q\leq n\})$ (previously done by Diaconis and Shahshahani \cite{Diaconis}) and  $\mathrm{Cay}(S_n,\{(1,q)\mid 2\leq q\leq n\})$ (previously obtained by Flatto, Odlyzko and Wales \cite[Theorem 3.7]{Flatto}). We show that a recent conjecture of Dai \cite{Dai2} is true as a consequence of Aldous'  theorem and we discuss some related questions and open problems.

\section{Main tools}\label{s-2}
Let  $G$ be a graph on $n$ vertices. The vertex partition $\Pi:V(G)=V_1\cup V_2\cup\cdots\cup V_q$ is said to be an \emph{equitable partition} of $G$ if every vertex of $V_i$ has  the same number (denoted by $b_{ij}$) of neighbors in $V_j$, for all $i,j\in\{1,2,\ldots,q\}$. The matrix $B_\Pi=(b_{ij})_{q\times q}$ is  the \emph{quotient matrix} of $G$ with respect to $\Pi$, and the $n\times q$ matrix $\chi_\Pi$ whose columns are the characteristic  vectors of $V_1,\ldots,V_q$ is  the \emph{characteristic matrix} of  $\Pi$.

\begin{lem}[Brouwer and Haemers \cite{Brouwer}, p. 30;  Godsil and Royle \cite{Godsil1}, pp. 196--198]\label{quotient} 
Let $G$ be a graph with adjacency matrix $A(G)$, and let $\Pi:V(G)=V_1\cup V_2\cup\cdots\cup V_q$ be an equitable partition of $G$ with quotient matrix $B_{\Pi}$. Then the eigenvalues of $B_\Pi$ are also eigenvalues of $A(G)$. Furthermore, $A(G)$ has the following two kinds of eigenvectors:
\begin{enumerate}[(i)]
\vspace{-0.2cm}
\item the eigenvectors in the column space of $\chi_\Pi$, and the corresponding eigenvalues coincide
with the eigenvalues of $B_\Pi$;
\vspace{-0.2cm}
\item the eigenvectors orthogonal to the columns of $\chi_\Pi$, i.e., those eigenvectors that sum to zero on each block $V_i$ for $1\leq i\leq q$.
\end{enumerate}
\end{lem}

For regular graphs, we have the following  useful result.

\begin{thm}\label{reg-thm}
Let $G$ be a $r$-regular graph, and let $\lambda$ ($\lambda\neq r$) be an eigenvalue of $G$.  If $G$ has  an eigenvector $f$ with respect to  $\lambda$ and a vertex partition $\Pi:V(G)=V_1\cup V_2\cup\cdots\cup V_q$ such that $G[V_i]$ is $r_1$-regular ($r_1\leq r$)  and $f$ sums to zero on $V_i$ for all $i\in\{1,2,\ldots,q\}$, then
$$\lambda\leq \max_{1\leq i\leq q}\lambda_2(G[V_i])+\lambda_2(G_1),$$
where $G_1$ is the $(r-r_1)$-regular graph obtained from $G$ by removing all edges in $\cup_{i=1}^qE(G[V_i])$.
\end{thm}
\begin{proof}
By assumption, the induced subgraphs $G[V_i]$ share the same degree $r_1$, so $G_1$ is  $(r-r_1)$-regular because $G$ is $r$-regular.  Also, the eigenvector  $f$  of $\lambda$ sums to zero on $V_i$ for  each $i$. Set $E_1=\cup_{i=1}^q E(G[V_i])$ and $E_2=E(G)\setminus E_1=E(G_1)$.  By the Rayleigh quotient, we obtain
\begin{equation}\label{eq-1}
\begin{aligned}
\lambda&=\frac{f^TA(G)f}{f^Tf}\\
&=\frac{\displaystyle 2\sum_{\{x,y\}\in E(G)}f(x)f(y)}{\displaystyle \sum_{x\in V(G)}f(x)^2}\\
&=\frac{\displaystyle 2\sum_{\{x,y\}\in E_1}f(x)f(y)}{\displaystyle \sum_{x\in V(G)}f(x)^2}+\frac{\displaystyle 2\sum_{\{x,y\}\in E_2}f(x)f(y)}{\displaystyle \sum_{x\in V(G)}f(x)^2}.
\end{aligned}
\end{equation}
For the first term, we have
\begin{equation}\label{eq-2}
\begin{aligned}
\frac{\displaystyle 2\sum_{\{x,y\}\in E_1}f(x)f(y)}{\displaystyle \sum_{x\in V(G)}f(x)^2}&=\frac{\displaystyle \sum_{i=1}^q2\sum_{\{x,y\}\in E(G[V_i])}f(x)f(y)}{\displaystyle \sum_{i=1}^q\sum_{x\in V_i}f(x)^2}\\
&\leq \max_{\begin{smallmatrix}1\leq i\leq q\\f|_{V_i}\neq 0\end{smallmatrix}}\frac{\displaystyle 2\sum_{\{x,y\}\in E(G[V_i])}f(x)f(y)}{\displaystyle \sum_{x\in V_i}f(x)^2}\\
&= \max_{\begin{smallmatrix}1\leq i\leq q\\f|_{V_i}\neq 0\end{smallmatrix}}\frac{f|_{V_i}^TA(G[V_i])f|_{V_i}}{f|_{V_i}^Tf|_{V_i}}\\
&\leq \max_{\begin{smallmatrix}1\leq i\leq q\\f|_{V_i}\neq 0\end{smallmatrix}}\max_{g\perp \mathbf{1}_{V_i}}\frac{g^TA(G[V_i])g}{g^Tg}\\
&=\max_{\begin{smallmatrix}1\leq i\leq q\\f|_{V_i}\neq 0\end{smallmatrix}}\lambda_2(G[V_i])\\
&\leq \max_{1\leq i\leq q}\lambda_2(G[V_i]),
\end{aligned}
\end{equation}
where $f|_{V_i}$ is the restriction  of $f$ on $V_i$, $\mathbf{1}_{V_i}$ is the all ones vector on $V_i$, and  the second inequality follows from  $\sum_{x\in V_i}f(x)=0$ ($1\leq i\leq q$). For the second term, since $G_1$ is regular and $f$ is orthogonal to the all ones vector $\mathbf{1}$, we have
\begin{equation}\label{eq-3}
\hspace{1cm}
\begin{aligned}
\frac{\displaystyle 2\sum_{\{x,y\}\in E_2}f(x)f(y)}{\displaystyle \sum_{x\in V(G)}f(x)^2}=\frac{f^TA(G_1)f}{f^Tf}\leq\max_{h\perp\mathbf{1}}\frac{h^TA(G_1)h}{h^Th}=\lambda_2(G_1).
\end{aligned}
\end{equation}
Combining  (\ref{eq-1}),  (\ref{eq-2}) and  (\ref{eq-3}), we conclude that
\begin{equation*}
\lambda\le \max_{1\leq i\leq q}\lambda_2(G[V_i])+\lambda_2(G_1),
\end{equation*}
and the result follows.
\end{proof}

If the partition $\Pi:V(G)=V_1\cup V_2\cup\cdots\cup V_q$ is exactly an equitable partition of $G$ with quotient matrix $B_\Pi$, then the eigenvectors of $G$ with respect to those eigenvalues other than that  of $B_\Pi$  must sum to zero on each $V_i$ by Lemma \ref{quotient}. From  Theorem \ref{reg-thm} one can immediately deduce the following result.

\begin{cor}\label{reg-cor}
Let $G$ be a $r$-regular graph. Assume that $\Pi:V(G)=V_1\cup V_2\cup\cdots\cup V_q$ is an equitable partition of $G$ whose quotient matrix $B_\Pi$ has constant diagonal entries. Then, for any eigenvalue  $\lambda$ of $G$ that is not that of  $B_\Pi$, we have
\begin{equation*}
\lambda\le \max_{1\leq i\leq q}\lambda_2(G[V_i])+\lambda_2(G_1),
\end{equation*}
where $G_1$ is the graph obtained from $G$ by removing all edges in $\cup_{i=1}^qE(G[V_i])$.
\end{cor}

Here we give an example to show how to use the result of Corollary \ref{reg-cor}.

\begin{exam}\label{Expander}
\emph{Let $H_1,H_2$ be two connected $k$-regular graphs on $n$ vertices. Let $G$ be the  graph (not unique) obtained from $H_1\cup H_2$ by adding  some new edges between $H_1$ and $H_2$ such that these edges  form a $r$-regular bipartite graph $G_1$ ($G_1$ is easy to  construct, cf. \cite{Huang2}, Lemma 3.2). Clearly, $G$ is a  connected $(k+r)$-regular graph. Let $V_1$ and $V_2$ be the vertex subsets of $G$ corresponding to $H_1$ and $H_2$, respectively. Then $V(G)=V_1\cup V_2$ is clearly an equitable partition of $G$ with quotient matrix
$$
B_\Pi=
\left[
\begin{matrix}
k&r\\
r&k
\end{matrix}
\right].
$$
Since  $\lambda_2(G_1)\leq r$, each eigenvalue of $G$ not belonging to $B_\Pi$ is bounded above by
$\max\{\lambda_2(H_1),\lambda_2(H_2)\}+r$ according to Corollary \ref{reg-cor}. As $\lambda_2(B_\Pi)=k-r$, we conclude that $$k-r\leq \lambda_2(G)\leq \max\{\max\{\lambda_2(H_1),\lambda_2(H_2)\}+r,k-r\}.$$
Note that the above bounds could be tight. Take $H_1=H_2=Q_n$, the $n$-dimensional hypercube, and let $G$ be the graph (not unique) obtained from $H_1\cup H_2$ by adding a perfect matching between $H_1$ and $H_2$ (such graphs contain the $(n+1)$-dimensional locally twisted cubes, cf. \cite{Yang}).  Since $\lambda_2(Q_n)=n-2$ (cf. \cite{Brouwer}, p. 19), we have 
$$n-1\leq \lambda_2(G)\leq \max\{\lambda_2(Q_n)+1,n-1\}=n-1,$$
and thus $\lambda_2(G)=n-1$, which attains the lower bound. Also,  the Cartesian product $C_n\square K_2$, which can be regarded as the graph obtained by adding a perfect matching between two copies of $C_n$, has second eigenvalue $2\cos\frac{2\pi}{n}+1=\lambda_2(C_n)+1$, and so attains the upper bound.
}
\end{exam}

By using Theorem \ref{reg-thm}, in what follows, we focus on providing upper bounds for some special eigenvalues of Cayley graphs. Before doing this, we need to do some preparatory work. First  of all, we give the following useful result, which  suggests that each Cayley graph has an equitable partition derived from left coset decomposition.
\begin{lem}\label{Cay-equitable}
Let $\Gamma$ be a finite group, and let $\mathrm{Cay}(\Gamma,T)$ be a Cayley graph of $\Gamma$. Then the set of left cosets of any subgroup $\Theta$ of $\Gamma$ gives an equitable partition of $\mathrm{Cay}(\Gamma,T)$.
\end{lem}
\begin{proof}
Suppose that  $\Pi: \Gamma=\gamma_1\Theta\cup\gamma_2\Theta\cup\cdots\cup \gamma_k\Theta$ is the left coset decomposition of $\Gamma$ with respect to $\Theta$, where $k=|\Gamma|/|\Theta|$ and $\gamma_1,\ldots,\gamma_k$ are the representation elements.  Clearly, $\Pi$ is a vertex partition of $\mathrm{Cay}(\Gamma,T)$. For any $\gamma\in \gamma_i\Theta$, we have $\gamma=\gamma_i\theta$ for some $\theta\in\Theta$, and therefore
\begin{equation*}
|N(\gamma)\cap \gamma_j\Theta|=|N(\gamma_i\theta)\cap \gamma_j\Theta|=|(T\gamma_i\theta)\cap \gamma_j\Theta|=|T\cap (\gamma_j\Theta\theta^{-1}\gamma_i^{-1})|=|T\cap (\gamma_j\Theta\gamma_i^{-1})|,
\end{equation*}
which is  independent on the choice of $\gamma\in \gamma_i\Theta$. Thus $\Pi$ is exactly an equitable partition of $\mathrm{Cay}(\Gamma,T)$, and the result follows.
\end{proof}

Let $\Omega$ be a nonempty set, and let $\Gamma$ be a group acting on $\Omega$. We say that the  action of $\Gamma$ on $\Omega$ ($|\Omega|\geq s$) is \emph{$s$-transitive} if  for all pairwise distinct $x_1, \ldots, x_s\in \Omega$ and pairwise distinct $y_1, \ldots, y_s\in \Omega$ there exists some $\gamma\in \Gamma$ such that $x_i^\gamma=y_i$ for $1 \leq i \leq  s$. Clearly, a $s$-transitive action is always $t$-transitive for any $t<s$. In particular, we say that the action is \emph{transitive}  if it is $1$-transitive. As usual, we denote by $\Gamma_x=\{\gamma\in \Gamma\mid x^\gamma=x\}$  the \emph{stabilizer subgroup} of $\Gamma$ with respect to $x\in \Omega$.

Now suppose that $\Gamma$ is a finite group acting transitively on  $[n]=\{1,2,\ldots,n\}$. For each fixed $i\in[n]$, we have  $|\Gamma|/|\Gamma_i|=n$ by the orbit-stabilizer theorem, and furthermore, we see that $\Gamma$ has left coset decomposition
\begin{equation}\label{eq-4}
\Pi_i:\Gamma=\gamma_{1,i}\Gamma_i\cup \gamma_{2,i}\Gamma_i\cup \cdots\cup \gamma_{n,i}\Gamma_i=\Gamma_{1,i}\cup \Gamma_{2,i}\cup \cdots \cup \Gamma_{n,i},
\end{equation}
where $\gamma_{j,i}$ is an arbitrary element in $\Gamma$ that maps $j$ to $i$ and
\begin{equation*}\label{eq-5}
\Gamma_{j,i}=\gamma_{j,i}\Gamma_i=\{\gamma\in \Gamma\mid j^\gamma=i\},
\end{equation*}
for all $j\in [n]$. Clearly, $|\Gamma_{j,i}|=|\Gamma_i|=|\Gamma|/n$.

Let $G=\mathrm{Cay}(\Gamma,T)$ be a Cayley graph of $\Gamma$. According to Lemma \ref{Cay-equitable}, for each $i\in [n]$, the left coset decomposition $\Pi_i$ given in  (\ref{eq-4}) is an equitable partition of $G$ with quotient matrix $B_{\Pi_i}=(b_{st})_{n\times n}$, where
\begin{equation}\label{eq-6}
b_{st}=|T\cap \gamma_{t,i}\Gamma_i\gamma_{s,i}^{-1}|=|T\cap \Gamma_{t,s}|
\end{equation}
is exactly the number of elements in $T$ mapping $t$ to $s$. Since $b_{st}=|T\cap \Gamma_{t,s}|$ is independent on the choice of $i$, all the equitable partitions $\Pi_i$ share the same quotient matrix. For this reason,  we use  $B_{\Pi}$ instead of $B_{\Pi_i}$. Also, by counting the edges between $\Gamma_{s,i}$ and $\Gamma_{t,i}$ in two ways, we obtain $b_{st}\cdot|\Gamma_{s,i}|=b_{ts}\cdot|\Gamma_{t,i}|$, which implies that $b_{st}=b_{ts}$ because $|\Gamma_{s,i}|=|\Gamma_{t,i}|=|\Gamma|/n$. Therefore, $B_{\Pi}=(b_{st})_{n\times n}$ is symmetric.

For any fixed $k\in [n]$, we also can partition the vertex set of $G$  as another form
\begin{equation}\label{eq-7}
\Pi'_k:\Gamma=\Gamma_{k,1}\cup \Gamma_{k,2}\cup \cdots\cup \Gamma_{k,n},
\end{equation}
which is exactly the right coset decomposition of $\Gamma$ with respect to $\Gamma_k$. In general, $\Pi'_k$ is not an equitable partition of $G$. As in Theorem \ref{reg-thm}, we can   decompose the edge set of $G$ into $E(G)=E_1\cup E_2$, where $E_1=\cup_{i=1}^n E(G[\Gamma_{k,i}])$ and $E_2=E(G)\setminus E_1$. Let $G_1$ denote the spanning subgraph of $G$ with edge set $E_2$. The following lemma determines the structure of $G_1$ and  $G[\Gamma_{k,i}]$ for all $i\in [n]$.

\begin{lem}\label{structure-lem}
For any fixed $k\in [n]$, we have
\begin{enumerate}[(i)]
\item $G[\Gamma_{k,i}]\cong \mathrm{Cay}(\Gamma_k,T\cap \Gamma_k)$ for all $i\in [n]$;
\item $G_1=\mathrm{Cay}(\Gamma,T\setminus (T\cap \Gamma_k))$.
\end{enumerate}
\end{lem}
\begin{proof}
For (i), the corresponding isomorphism  can be defined as
$$\begin{aligned}
\phi: \Gamma_{k,i}=\gamma_{k,i}\Gamma_i&\to \gamma_{k,i}\Gamma_i\gamma_{k,i}^{-1}=\Gamma_k\\
\gamma_{k,i}\gamma&\mapsto \gamma_{k,i}\gamma\gamma_{k,i}^{-1}, ~~\forall \gamma\in\Gamma_i.
\end{aligned}
$$
Clearly, $\phi$ is one-to-one and onto. Furthermore, we have
$$\begin{aligned}
\{\gamma_{k,i}\gamma,\gamma_{k,i}\gamma'\}\in E(G[\Gamma_{k,i}])&\Longleftrightarrow\gamma_{k,i}\gamma'(\gamma_{k,i}\gamma)^{-1}\in T\\
&\Longleftrightarrow \gamma_{k,i}\gamma'\gamma^{-1}\gamma_{k,i}^{-1}\in T\cap \gamma_{k,i}\Gamma_i\gamma_{k,i}^{-1}=T\cap \Gamma_k\\
&\Longleftrightarrow \gamma_{k,i}\gamma'\gamma_{k,i}^{-1}(\gamma_{k,i}\gamma\gamma_{k,i}^{-1})^{-1}\in T\cap \Gamma_k\\
&\Longleftrightarrow \{\gamma_{k,i}\gamma\gamma_{k,i}^{-1},\gamma_{k,i}\gamma'\gamma_{k,i}^{-1}\}\in E(\mathrm{Cay}(\Gamma_k,T\cap \Gamma_k)),
\end{aligned}
$$
and so (i) follows. Now we consider (ii). Clearly, $G_1[\Gamma_{k,i}]$ is an empty graph for all $i\in [n]$. For any $\gamma_{k,i}\gamma\in \Gamma_{k,i}=\gamma_{k,i}\Gamma_i$ and $\gamma_{k,j}\gamma'\in \Gamma_{k,j}=\gamma_{k,j}\Gamma_j$ ($i\neq j$), we have
$\{\gamma_{k,i}\gamma,\gamma_{k,j}\gamma'\}\in E(G_1)$ if and only if $\gamma_{k,j}\gamma'(\gamma_{k,i}\gamma)^{-1}\in T$, which is the case if and only if $\gamma_{k,j}\gamma'(\gamma_{k,i}\gamma)^{-1}\in T\setminus (T\cap \Gamma_k)$ because $\gamma_{k,j}\gamma'(\gamma_{k,i}\gamma)^{-1}=\gamma_{k,j}\gamma'\gamma^{-1}\gamma_{k,i}^{-1}\not\in\Gamma_k$ due to $i\neq j$. Therefore, each edge of $G_1$ comes from $T\setminus (T\cap \Gamma_k)$. Conversely, $T\setminus (T\cap \Gamma_k)$ can only be used to produce the edges in $E(G_1)=E_2$ because each edge in $E_1=\cup_{i=1}^n E(G[\Gamma_{k,i}])$ comes from $T\cap T_k$. This proves (ii).
\end{proof}

Now we are in a position to give the main result of this section, which provides upper bounds for some special eigenvalues of Cayley graphs.

\begin{thm}\label{Cay-thm}
Let $\Gamma$ be a finite group acting transitively on $[n]=\{1,2,\ldots,n\}$, and let $G=\mathrm{Cay}(\Gamma,T)$ be a  Cayley graph of $\Gamma$. Then the left coset decomposition $\Pi_i$ of $\Gamma$ given in  (\ref{eq-4}) leads to an equitable partition of $G$, and the corresponding quotient matrix $B_{\Pi}=B_{\Pi_i}$ is symmetric and independent on the choice of $i$. Moreover, if $\lambda$ is an eigenvalue of $G$ other than that of $B_\Pi$, then, for each $k\in [n]$, we have
$$
\lambda\leq \lambda_2(\mathrm{Cay}(\Gamma_k,T\cap \Gamma_k))+\lambda_2(\mathrm{Cay}(\Gamma,T\setminus (T\cap \Gamma_k))),
$$
where $\Gamma_k$ is the stabilizer subgroup of $\Gamma$ with respect to $k$.
\end{thm}
\begin{proof}
From the above arguments, it suffices to prove the second part of the theorem. Let $f$ be an arbitrary eigenvector of $G$ with respect to  $\lambda$. Since $\Pi_i$ is an equitable partition of $G$ for each $i$, we see that $f$ must sum to zero on $\Gamma_{j,i}$ for all $i,j\in [n]$ by Lemma \ref{quotient}. For any fixed $k\in [n]$, let $\Pi'_k$ be the vertex partition of $G$ given in (\ref{eq-7}). In particular, we have that $f$ sums to zero on $\Gamma_{k,i}$ for all $i\in [n]$.  By Lemma \ref{structure-lem}, all these induced subgraphs $G[\Gamma_{k,i}]$ ($i\in[n]$) are isomorphic to $\mathrm{Cay}(\Gamma_k,T\cap \Gamma_k)$, and so share the same degree $|T\cap \Gamma_k|$. Let $G_1$ be the graph obtained from  $G$  by removing all edges in $\cup_{i=1}^n E(G[\Gamma_{k,i}])$. Note that $G_1\cong \mathrm{Cay}(\Gamma,T\setminus (T\cap \Gamma_k))$ again by Lemma \ref{structure-lem}. Then, by applying Theorem \ref{reg-thm} to the vertex partition $\Pi_k'$, we obtain
\begin{equation*}
\begin{aligned}
\lambda&\leq \max_{1\leq i\leq n}\lambda_2(G[\Gamma_{k,i}])+\lambda_2(G_1)\\
&=\lambda_2(\mathrm{Cay}(\Gamma_k,T\cap \Gamma_k))+\lambda_2(\mathrm{Cay}(\Gamma,T\setminus (T\cap \Gamma_k))).
\end{aligned}
\end{equation*}
By the arbitrariness of $k\in [n]$, our result follows.
\end{proof}

It is worth mentioning that Theorem \ref{Cay-thm} provides for us a recursive method to determine the second eigenvalue of  the connected Cayley graph $G=\mathrm{Cay}(\Gamma,T)$. Indeed, by Lemma \ref{quotient},  all eigenvalues of $B_\Pi$ are also that of $G$,  so we have
$\lambda_2(G)\geq \lambda_2(B_\Pi)$. Therefore, if there exists some $k\in [n]$ such that
\begin{equation}\label{eq-8}
\lambda_2(\mathrm{Cay}(\Gamma_k,T\cap \Gamma_k))+\lambda_2(\mathrm{Cay}(\Gamma,T\setminus (T\cap \Gamma_k)))\leq \lambda_2(B_\Pi),
\end{equation}
then we may conclude  that $\lambda_2(G)=\lambda_2(B_\Pi)$ by Theorem \ref{Cay-thm}. Thus the problem is reduced to determining the exact value of $\lambda_2(\mathrm{Cay}(\Gamma_k,T\cap \Gamma_k))$ and $\lambda_2(\mathrm{Cay}(\Gamma,T\setminus (T\cap \Gamma_k)))$, which reminds us that the way of induction could be applied.

In the next section, we shall see that if $\Gamma$ and $T$ satisfy some additional conditions  then the problem of proving $\lambda_2(G)= \lambda_2(B_\Pi)$ can be reduced to that of verifying the result for some small graphs.

\section{Normal Cayley graphs}\label{s-3}

For a finite group $\Gamma$,  the \emph{conjugacy class} of $\gamma\in\Gamma$ is defined as the set  $\mathcal{C}_\gamma=\{\sigma^{-1}\gamma\sigma\mid \sigma\in\Gamma\}$. Recall that a Cayley graph $\mathrm{Cay}(\Gamma,T)$ is said to be normal if $T$ is closed under conjugation, that is, $T$ is the disjoint union of some conjugacy classes of $\Gamma$. It is well known that the eigenvalues of a normal Cayley graph can be expressed in terms of the irreducible characters of $\Gamma$.
\begin{thm}[\cite{Babai,Lubotzky,Murty}]\label{normal-thm-0}
The eigenvalues of a normal Cayley graph $\mathrm{Cay}(\Gamma,T)$ are given by
$$
\lambda_\chi=\frac{1}{\chi(1)}\sum_{\tau\in T}\chi(\tau),
$$
where $\chi$ ranges over all the irreducible characters of $\Gamma$. Moreover, the multiplicity of $\lambda_\chi$
is $\chi(1)^2$.
\end{thm}

However, it is often difficult  to identify the second eigenvalues of   normal Cayley graphs from Theorem \ref{normal-thm-0}.
In this section, by using Theorem \ref{Cay-thm}, we  reduce the problem of determining  the second eigenvalues  of normal Cayley graphs of high transitive groups  to that of verifying the result for some smaller graphs.

From now on, we always assume that $\Gamma$ acts transitively on $[n]$, and that  $G=\mathrm{Cay}(\Gamma,T)$ is a connected normal Cayley graph of $\Gamma$, i.e., $T$ is a generating subset of $\Gamma$ which is also closed under conjugation. In order to use Theorem \ref{Cay-thm} recursively, we set $T_0=T$, $G_0=\mathrm{Cay}(\Gamma,T_0)=G$, and for $k=1,2,\ldots,n$, we define
\begin{equation}\label{eq-13}
\begin{aligned}
&G_k=\mathrm{Cay}(\Gamma,T_k)~\mbox{with}~T_{k}=T_{k-1}\setminus (T_{k-1}\cap \Gamma_{k});\\
&H_k=\mathrm{Cay}(\Gamma_k,R_k)~\mbox{with}~R_{k}=T_{k-1}\cap \Gamma_{k}.
\end{aligned}
\end{equation}
We see that both  $G_k$ and $H_k$ are subgraphs of $G_{k-1}$, and furthermore,  by regarding $T_{k-1}$ as $T$ in Lemma  \ref{structure-lem}, we have
\begin{claim}\label{claim-1}
The edge set of $G_{k-1}$ ($k\geq 1$) can be decomposed into that of $G_k$ and $n$ copies of $H_k$.
\end{claim}

Note that $T_1=T\setminus (T\cap \Gamma_1)$ consists of those elements in $T$ moving $1$, $T_2=T_1\setminus (T_1\cap \Gamma_2)$ consists of those elements in $T_1$ moving $2$, i.e., those elements in $T$ moving both $1$ and $2$, and so on. Thus we have

\begin{claim}\label{claim-2}
For each $k\geq 1$, $T_k$ is the set of $\tau\in T$ satisfying $\{1,2,\ldots,k\}\subseteq  \mathrm{supp}(\tau)$, i.e., $T_k=T\setminus(T\cap (\cup_{i=1}^k \Gamma_{i}))$, and thus $R_k=T_{k-1}\cap \Gamma_{k}$ is the set of elements in $T$ moving   $1,2,\ldots,{k-1}$ but fixing  $k$.
\end{claim}

Note that  $\Gamma$ acts transitively on $[n]$. For  $0\leq k\leq n$,  from  Theorem \ref{Cay-thm} and  (\ref{eq-6})  we see that the left coset decompositions $\Pi_i$ ($i\in[n]$) of $\Gamma$ given in  (\ref{eq-4})  are equitable partitions of $G_k=\mathrm{Cay}(\Gamma,T_k)$ which share the same symmetric quotient matrix
\begin{equation}\label{eq-10}
B_\Pi^{(k)}=(b_{st}^{(k)})_{n\times n},~\mbox{where $b_{st}^{(k)}=|T_k\cap \Gamma_{t,s}|$}.
\end{equation}
In particular, $B_\Pi^{(0)}=B_\Pi$.

To achieve our goal, we need to  determine the second eigenvalue of  $B_\Pi^{(k)}$ ($k\geq 0$). 

\begin{lem}\label{quotient-lem2}
Let $G_k=\mathrm{Cay}(\Gamma,T_k)$ ($k\geq 0$) be the graph defined in  (\ref{eq-13}), and $B_\Pi^{(k)}$ the quotient matrix of $G_k$ defined in (\ref{eq-10}).  If $\Gamma$ acts $(k+2)$-transitively on $[n]$, then   $\lambda_2(B_\Pi^{(k)})=|T_k\cap \Gamma_{k+1}|-|T_k\cap \Gamma_{k+2,k+1}|$.
\end{lem}
\begin{proof}
First suppose $k=0$. According to  (\ref{eq-10}), we have $B_\Pi^{(0)}=(b_{st}^{(0)})_{n\times n}$, where $b_{st}^{(0)}=|T_0\cap \Gamma_{t,s}|$. Since $\Gamma$ acts $2$-transitively on $[n]$, for any $s\in[n]$, there exists some $\sigma\in\Gamma$ such that $\sigma$ maps $s$ to $1$. Considering that $T_0=T$ is closed under conjugation, we have $b_{ss}^{(0)}=|T_0\cap \Gamma_{s,s}|=|T_0\cap \Gamma_{s}|=|\sigma^{-1}(T_0\cap \Gamma_{s})\sigma|=|(\sigma^{-1}T_0\sigma)\cap (\sigma^{-1}\Gamma_{s}\sigma)|=|T_0\cap \Gamma_{1}|=b_{11}^{(0)}$.   Similarly, for any two distinct $s,t\in [n]$, there exists some $\sigma$ in $\Gamma$ mapping $s$ to $1$ and $t$ to $2$ by the $2$-transitivity of $\Gamma$ acting on $[n]$.  Then $b_{st}^{(0)}=|T_0\cap \Gamma_{t,s}|=|\sigma^{-1}(T_0\cap \Gamma_{t,s})\sigma|=|(\sigma^{-1}T_0\sigma)\cap (\sigma^{-1}\Gamma_{t,s}\sigma)|=|T_0\cap \Gamma_{t^\sigma,s^\sigma}|=|T_0\cap \Gamma_{2,1}|=b_{12}^{(0)}$. Combining these results, we have
$$B_\Pi^{(0)}=b_{11}^{(0)}\cdot I_n+b_{12}^{(0)}\cdot (J_n-I_n).$$
Thus the quotient matrix $B_\Pi^{(0)}$ has eigenvalues $|T|=b_{11}^{(0)}+(n-1)\cdot b_{12}^{(0)}$ of multiplicity one and $b_{11}^{(0)}-b_{12}^{(0)}$ of multiplicity $n-1$. Therefore,  $\lambda_2(B_\Pi^{(0)})=b_{11}^{(0)}-b_{12}^{(0)}=|T_0\cap \Gamma_1|-|T_0\cap \Gamma_{2,1}|$, and our result follows.

Now suppose $k\geq 1$. By definition, we see that $T_k=T\setminus(T\cap (\cup_{l=1}^k \Gamma_l))$. We claim that if $\sigma$ is an element in $\Gamma$ fixing $\{1,2,\ldots,k\}$ setwise then   $\sigma^{-1}T_k\sigma=T_k$. Indeed, we have $\sigma^{-1}T_k\sigma=(\sigma^{-1}T\sigma) \setminus((\sigma^{-1}T\sigma)\cap (\cup_{l=1}^k \sigma^{-1}\Gamma_l\sigma))=T\setminus(T\cap (\cup_{l=1}^k\Gamma_{l^{\sigma}}))=T\setminus(T\cap (\cup_{l=1}^k\Gamma_l))=T_k$, as required.

We shall determine all eigenvalues of $B_\Pi^{(k)}$. According to   (\ref{eq-10}), we see that $B_\Pi^{(k)}=(b_{st}^{(k)})$, where $b_{st}^{(k)}=|T_k\cap \Gamma_{t,s}|$. For $1\leq s\leq k$, we have $b_{ss}^{(k)}=|T_k\cap \Gamma_{s,s}|=0$ because $T_k$ must move $s$ but $\Gamma_{s,s}=\Gamma_s$ does not. For  $k+1\leq s\leq n$, by the $(k+2)$-transitivity   of $\Gamma$ acting on $[n]$, there is a $\sigma\in\Gamma$  fixing $\{1,2,\ldots,k\}$ setwise but moving $s$ to $k+1$. Then  $\sigma^{-1}T_k\sigma=T_k$  and $\sigma^{-1}\Gamma_s\sigma=\Gamma_{k+1}$ by above arguments,  and thus $b_{ss}^{(k)}=|T_k\cap \Gamma_{s,s}|=|T_k\cap \Gamma_{s}|=|\sigma^{-1}(T_k\cap \Gamma_{s})\sigma|=|(\sigma^{-1}T_k\sigma) \cap (\sigma^{-1}\Gamma_s\sigma)|=|T_k\cap \Gamma_{k+1}|=b_{k+1,k+1}^{(k)}$.  For $1\leq s<t\leq k$ (if $k\geq 2$), again by the   $(k+2)$-transitivity, we can choose $\sigma\in\Gamma$  such that $\sigma$ moves $t$ to $2$ and $s$ to $1$ but fixes  $\{1,2,\ldots,k\}$ setwise. Then we see that $b_{st}^{(k)}=|T_k\cap \Gamma_{t,s}|=|\sigma^{-1}(T_k\cap \Gamma_{t,s})\sigma|=|(\sigma^{-1}T_k\sigma)\cap (\sigma^{-1}\Gamma_{t,s}\sigma)|=|T_k\cap \Gamma_{2,1}|=b_{12}^{(k)}$. For $1\leq s\leq k$ and  $k+1\leq t\leq n$, there also exists  some $\sigma$ in $\Gamma$  mapping $s$ to $1$, $t$ to $k+1$ but fixing  $\{1,2,\ldots,k\}$ setwise, thus we get $b_{st}^{(k)}=|T_k\cap \Gamma_{t,s}|=|\sigma^{-1}(T_k\cap \Gamma_{t,s})\sigma|=|T_k\cap \Gamma_{k+1,1}|=b_{1,k+1}^{(k)}$. For $k+1\leq s<t\leq n$, we take $\sigma\in \Gamma$ such that $\sigma$ maps $s$ to $k+1$ and $t$ to $k+2$ but fixes  $\{1,2,\ldots,k\}$ setwise. Then $b_{st}^{(k)}=|T_k\cap \Gamma_{t,s}|=|\sigma^{-1}(T_k\cap \Gamma_{t,s})\sigma|=|T_k\cap \Gamma_{k+2,k+1}|=b_{k+1,k+2}^{(k)}$. Concluding these results, we have
\begin{equation*}
b_{st}^{(k)}=b_{ts}^{(k)}=\left\{
\begin{array}{ll}
0,& \mbox{if $1\leq s=t\leq k$;}\\
|T_k\cap \Gamma_{k+1}|=b_{k+1,k+1}^{(k)},& \mbox{if $k+1\leq s=t\leq n$;}\\
|T_k\cap \Gamma_{2,1}|=b_{1,2}^{(k)},& \mbox{if $1\leq s<t\leq k$ (for $k\geq 2$);}\\
|T_k\cap \Gamma_{k+1,1}|=b_{1,k+1}^{(k)},& \mbox{if $1\leq s\leq k$, $k+1\leq t\leq n$;}\\
|T_k\cap \Gamma_{k+2,k+1}|=b_{k+1,k+2}^{(k)},& \mbox{if $k+1\leq s<t\leq n$.}
\end{array}
\right.
\end{equation*}
Therefore, the quotient matrix $B_\Pi^{(k)}$ can be written as
$$
B_\Pi^{(k)}=\left[
\begin{matrix}
b_{1,2}^{(k)}\cdot (J_k-I_k)&b_{1,k+1}^{(k)}\cdot J_{k\times (n-k)}\\
b_{1,k+1}^{(k)}\cdot J_{(n-k)\times k}& b_{k+1,k+1}^{(k)}\cdot I_{n-k}+b_{k+1,k+2}^{(k)}\cdot (J_{n-k}-I_{n-k})\\
\end{matrix}
\right].
$$
Take $f_1=(g_1^T,0^T)^T\in \mathbb{R}^n$  and  $f_2=(0^T, g_2^T)^T\in \mathbb{R}^n$, where $g_1\in\mathbb{R}^k$ and $g_2\in\mathbb{R}^{n-k}$ are two arbitrary vectors orthogonal to the all ones vector, respectively. One can easily verify that $B_\Pi^{(k)}f_1=-b_{1,2}^{(k)}\cdot f_1$ and $B_\Pi^{(k)}f_2=(b_{k+1,k+1}^{(k)}-b_{k+1,k+2}^{(k)})\cdot f_2$, so $-b_{1,2}^{(k)}$ and $b_{k+1,k+1}^{(k)}-b_{k+1,k+2}^{(k)}$ are eigenvalues of $B_\Pi^{(k)}$ with multiplicities at least $k-1$ and $n-k-1$, respectively. Also note that $|T_k|$ is always  an eigenvalue of $B_\Pi^{(k)}$ with the all ones vector  as its eigenvector because $G_k=\mathrm{Cay}(\Gamma,T_k)$ is $|T_k|$-regular. Thus there is just one eigenvalue, denoted by $\mu$, that is not known. By computing the trace of $B_\Pi^{(k)}$ in two ways, we obtain
$$
(n-k)\cdot b_{k+1,k+1}^{(k)}=
|T_k|-(k-1)\cdot b_{1,2}^{(k)}+(n-k-1)\cdot (b_{k+1,k+1}^{(k)}-b_{k+1,k+2}^{(k)})+\mu,
$$
which gives that
\begin{eqnarray*}
\mu&=&b_{k+1,k+1}^{(k)}+(n-k-1)\cdot b_{k+1,k+2}^{(k)}-(|T_k|-(k-1)\cdot b_{1,2}^{(k)})\\
&=&b_{k+1,k+1}^{(k)}+(n-k-1)\cdot b_{k+1,k+2}^{(k)}-(n-k)\cdot b_{1,k+1}^{(k)}\\
&=&b_{k+1,k+1}^{(k)}+(n-k-1)\cdot b_{k+1,k+2}^{(k)}-(n-k)\cdot b_{k+1,1}^{(k)}.
\end{eqnarray*}
Thus the eigenvalues of $B_\Pi^{(k)}$ are $|T|$, $-b_{1,2}^{(k)}$ (with multiplicity  $k-1$), $b_{k+1,k+1}^{(k)}-b_{k+1,k+2}^{(k)}$ (with multiplicity  $n-k-1$) and $\mu=b_{k+1,k+1}^{(k)}+(n-k-1)\cdot b_{k+1,k+2}^{(k)}-(n-k)\cdot b_{k+1,1}^{(k)}$.

Now we  prove  that $\lambda_2(B_\Pi^{(k)})=b_{k+1,k+1}^{(k)}-b_{k+1,k+2}^{(k)}$. Since $\lambda_1(B_\Pi^{(k)})=|T_k|$, it remains to compare the remaining  eigenvalues. To prove $b_{k+1,k+1}^{(k)}-b_{k+1,k+2}^{(k)}\geq \mu=b_{k+1,k+1}^{(k)}+(n-k-1)\cdot b_{k+1,k+2}^{(k)}-(n-k)\cdot b_{k+1,1}^{(k)}$, it suffices to show that  $b_{k+1,1}^{(k)}\geq b_{k+1,k+2}^{(k)}$. Indeed, by the $(k+2)$-transitivity  of $\Gamma$ acting on $[n]$, there exists some $\sigma\in \Gamma$  such that
$\sigma$ moves $1$ to $k+2$ but fixes $k+1$ and $\{2,\ldots,k\}$ setwise. Then $\sigma^{-1}T_k\sigma =(\sigma^{-1}T\sigma) \setminus((\sigma^{-1}T\sigma)\cap (\cup_{l=1}^k \sigma^{-1}\Gamma_l\sigma))=T\setminus(T\cap (\cup_{l=1}^k\Gamma_{l^{\sigma}}))=T\setminus(T\cap (\Gamma_{k+2}\cup (\cup_{l=2}^k\Gamma_l)))$, and so we obtain
\begin{equation}\label{eq-14}
\begin{aligned}
b_{k+1,1}^{(k)}&=|T_k\cap \Gamma_{1,k+1}|\\
&=|\sigma^{-1}(T_k\cap \Gamma_{1,k+1})\sigma|\\
&=|(\sigma^{-1}T_k\sigma)\cap (\sigma^{-1}\Gamma_{1,k+1}\sigma)|\\
&=|(T\setminus(T\cap (\Gamma_{k+2}\cup (\cup_{l=2}^k\Gamma_l))))\cap \Gamma_{k+2,k+1}|\\
&=|T\cap \Gamma_{k+2,k+1}|-|T\cap (\Gamma_{k+2}\cup (\cup_{l=2}^k\Gamma_l))\cap \Gamma_{k+2,k+1}|\\
&=|T\cap \Gamma_{k+2,k+1}|-|T\cap (\cup_{l=2}^k\Gamma_l)\cap \Gamma_{k+2,k+1}|,
\end{aligned}
\end{equation}
where the last equality follows from  $\Gamma_{k+2}\cap \Gamma_{k+2,k+1}=\emptyset$.
Also, we see that
\begin{equation}\label{eq-15}
b_{k+1,k+2}^{(k)}=|T_k\cap \Gamma_{k+2,k+1}|=|T\cap \Gamma_{k+2,k+1}|-|T\cap (\cup_{l=1}^k\Gamma_l)\cap \Gamma_{k+2,k+1}|.
\end{equation}
Combining  (\ref{eq-14}) and  (\ref{eq-15}) yields
\begin{equation*}
b_{k+1,1}^{(k)}-b_{k+1,k+2}^{(k)}=|T\cap (\cup_{l=1}^k\Gamma_l)\cap \Gamma_{k+2,k+1}|-|T\cap (\cup_{l=2}^k\Gamma_l)\cap \Gamma_{k+2,k+1}|\geq 0,
\end{equation*}
as required.  Now let us show that $b_{k+1,k+1}^{(k)}-b_{k+1,k+2}^{(k)}\geq -b_{1,2}^{(k)}$.  Since $-b_{1,2}^{(k)}$ is not an eigenvalue of $B_\Pi^{(k)}$ when $k=1$, we can suppose $k\geq 2$.  If  we can prove $b_{1,2}^{(k)}\geq b_{k+1,k+2}^{(k)}$, then the result follows because $b_{k+1,k+1}^{(k)}\geq 0$. As above,  by taking $\sigma\in \Gamma$ such that $\sigma$ maps $1$ to $k+1$ and $2$ to $k+2$ but fixes $\{3,\ldots,k\}$ setwise, we get
\begin{equation}\label{eq-16}
\begin{aligned}
b_{1,2}^{(k)}&=|T_k\cap \Gamma_{2,1}|\\
&=|\sigma^{-1}(T_k\cap \Gamma_{2,1})\sigma|\\
&=|(\sigma^{-1}T_k\sigma) \cap \Gamma_{k+2,k+1}|\\
&=|T\cap \Gamma_{k+2,k+1}|-|T\cap (\cup_{l=3}^{k+2}\Gamma_l)\cap \Gamma_{k+2,k+1}|\\
&=|T\cap \Gamma_{k+2,k+1}|-|T\cap (\cup_{l=3}^{k}\Gamma_l)\cap \Gamma_{k+2,k+1}|.
\end{aligned}
\end{equation}
Combining  (\ref{eq-15}) and  (\ref{eq-16}), we have
\begin{equation*}
b_{1,2}^{(k)}-b_{k+1,k+2}^{(k)}=|T\cap (\cup_{l=1}^k\Gamma_l)\cap \Gamma_{k+2,k+1}|-|T\cap (\cup_{l=3}^k\Gamma_l)\cap \Gamma_{k+2,k+1}|\geq 0,
\end{equation*}
and the result follows. Hence  we conclude that
$$\lambda_2(B_\Pi^{k})=b_{k+1,k+1}^{(k)}-b_{k+1,k+2}^{(k)}=|T_k\cap \Gamma_{k+1}|-|T_k\cap \Gamma_{k+2,k+1}|.$$

The proof is complete.
\end{proof}

Set
$$m=\max_{\tau\in T} |\mathrm{supp}(\tau)|.$$
If $m<n$, then  we claim that  $G_m=\mathrm{Cay}(\Gamma,T_m)$ is disconnected. Indeed, by the definition, $T_m$ consists of those $\tau\in T$ such that $\{1,2,\ldots,m\}\subseteq \mathrm{supp}(\tau)$. Since each element of $T$ has at most $m$ supports, we have $\mathrm{supp}(\tau)=\{1,2,\ldots,m\}$ for any $\tau\in T_m$, which implies that $T_m$ cannot generate $\Gamma$ due to $m<n$.

In the following,  we suppose further that  the action of $\Gamma$ on $[n]$ is $(m+a)$-transitive with $a\geq 1$. Under this assumption, it is clear that $n\geq m+a$, and so $m<n$, implying that $G_m$ is disconnected. Denote by
\begin{equation}\label{eq-17}
\Gamma^{(0)}=\Gamma  ~\mbox{and}~   \Gamma^{(i)}=\cap_{j=1}^{i}\Gamma_{n-j+1} ~\mbox{for $1\leq i\leq a-1$}.
\end{equation}
Indeed, $\Gamma^{(i)}$ ($1\leq i\leq a-1$) is just the subgroup of $\Gamma$ that fixes each point of $\{n-i+1,\ldots,n\}$. For this reason,  we can also regard $\Gamma^{(i)}$ as a group acting on $[n-i]=\{1,2,\ldots,n-i\}$. Moreover,  this action is $(m+a-i)$-transitive because $\Gamma$ acts $(m+a)$-transitively on $[n]$. For $0\leq i\leq a-1$, we define
\begin{equation}\label{eq-18}
\begin{aligned}
&G_{k,i}=\mathrm{Cay}(\Gamma^{(i)},T_k\cap \Gamma^{(i)})~\mbox{for}~0\leq k\leq m;\\
&H_{k,i}=\mathrm{Cay}(\Gamma^{(i)}\cap \Gamma_{k}, R_k\cap \Gamma^{(i)})~\mbox{for}~1\leq k\leq m,
\end{aligned}
\end{equation}
where $\Gamma^{(i)}$ is defined in  (\ref{eq-17}), and $T_k,R_k$ are given in  (\ref{eq-13}). By definition, $G_{k,0}=G_k=\mathrm{Cay}(\Gamma,T_k)$, $H_{k,0}=H_k=\mathrm{Cay}(\Gamma_k,R_k)$, and $G_{k,i}$ is the subgraph of both $G_{k-1,i}$ and $G_{k,i-1}$. As in   Claim \ref{claim-1}, the edge set of $G_{k-1,i}$  can  be decomposed into that of $G_{k,i}$ and $(n-i)$-copies of  $H_{k,i}$. Also,  for each fixed $i$, we see that  $T_0\cap \Gamma^{(i)}=T\cap \Gamma^{(i)}$ is  closed under conjugation in $\Gamma^{(i)}$, and $T_{k}\cap \Gamma^{(i)}$ is just the set of elements in $T\cap  \Gamma^{(i)}$ moving each point of $\{1,2,\ldots,k\}$ (similar as Claim \ref{claim-2}). Furthermore, since $n-i\geq m+a-i\geq m+1$, we claim that $T_m\subseteq\Gamma^{(i)}$ and that $G_{m,i}=\mathrm{Cay}(\Gamma^{(i)},T_m\cap \Gamma^{(i)})=\mathrm{Cay}(\Gamma^{(i)},T_m)$ is disconnected. In particular, we have
$\lambda_2(G_{m,i})=|T_m\cap \Gamma^{(i)}|=|T_m|$ for all $0\leq i\leq a-1$. Recall that $\Gamma^{(i)}$  acts  $(m+a-i)$-transitively ($m+a-i\geq m+1$) on $[n-i]$. According to  Lemma \ref{Cay-equitable} and the arguments in Section \ref{s-2}, every left coset decomposition of $\Gamma^{(i)}$ with respect to some stabilizer subgroup leads to an equitable partition of $G_{k,i}$, and all these equitable partitions share the same quotient matrix
$$B_\Pi^{(k,i)}=(b_{st}^{(k,i)})_{(n-i)\times (n-i)},~\mbox{where $b_{st}^{(k,i)}=|T_{k}\cap \Gamma^{(i)}\cap \Gamma_{t,s}|$}.$$
Clearly, $B_\Pi^{(k,0)}$ coincides with $B_\Pi^{(k)}$. For $0\leq k\leq m-1$, we have $k+2\leq m+1\leq m+a-i$, and so $\Gamma^{(i)}$  acts  $(k+2)$-transitively  on $[n-i]$. By applying Lemma  \ref{quotient-lem2} to $G_{k,i}$, we obtain
\begin{equation}\label{eq-19}
\lambda_2(B_\Pi^{(k,i)})=|T_{k}\cap \Gamma^{(i)}\cap \Gamma_{k+1}|-|T_{k}\cap \Gamma^{(i)}\cap \Gamma_{k+2,k+1}|,
\end{equation}
where $0\leq k\leq m-1$ and $0\leq i\leq a-1$.

Before giving the main result of this section, we need the following two lemmas.

\begin{lem}\label{recursive-lem}
Let $m$, $a$ and $B_\Pi^{(k,i)}$ be defined as above.
Assume that $a\geq 2$.  For $0\leq i\leq a-2$, we have
$$
\lambda_2(B_\Pi^{(k,i)})-\lambda_2(B_\Pi^{(k,i+1)})=
\left\{
\begin{array}{ll}
\lambda_2(B_\Pi^{(k+1,i)}),& \mbox{if $0\leq k\leq m-2$;}\\
|T_m|,& \mbox{if $k=m-1$.}\\
\end{array}
\right.
$$
\end{lem}
\begin{proof}
Since $\Gamma$ acts $(m+a)$-transitively on $[n]$, there exists some $\sigma_1,\sigma_2\in \Gamma$ such that $\sigma_1$ moves $k+1$ to $k+2$, $n-i$ to $k+1$, $\sigma_2$ moves $k+1$ to $k+2$, $k+2$ to $k+3$ and $n-i$ to $k+1$, and both of them fix $\{1,\ldots,k\}$ and $\{n-i+1,\ldots,n\}$ setwise. Then we have $\sigma_{j}^{-1}T_k\sigma_{j}=T_k$, $\sigma_{j}^{-1}\Gamma^{(i)}\sigma_{j}=\Gamma^{(i)}$ and $\sigma_j^{-1}\Gamma^{(i+1)}\sigma_j=\sigma_j^{-1}(\Gamma_{n-i}\cap \Gamma^{(i)})\sigma_j=\Gamma_{k+1}\cap \Gamma^{(i)}$ for $j=1,2$, which gives that
\begin{equation}\label{eq-20}
\left\{
\begin{aligned}
&\sigma_1^{-1}(T_k\cap \Gamma^{(i)}\cap \Gamma_{k+1})\sigma_1=T_k\cap \Gamma^{(i)}\cap \Gamma_{k+2};\\
&\sigma_1^{-1}(T_k\cap \Gamma^{(i+1)}\cap \Gamma_{k+1})\sigma_1=T_k\cap \Gamma_{k+1}\cap \Gamma^{(i)}\cap \Gamma_{k+2};\\
&\sigma_2^{-1}(T_k\cap \Gamma^{(i)}\cap \Gamma_{k+2,k+1})\sigma_2=T_k\cap \Gamma^{(i)}\cap \Gamma_{k+3,k+2};\\
&\sigma_2^{-1}(T_k\cap \Gamma^{(i+1)}\cap \Gamma_{k+2,k+1})\sigma_2=T_k\cap \Gamma_{k+1}\cap \Gamma^{(i)}\cap \Gamma_{k+3,k+2}.
\end{aligned}
\right.
\end{equation}
Also recall that $T_{k+1}=T_k\setminus(T_k\cap \Gamma_{k+1})$. According to  (\ref{eq-19}) and  (\ref{eq-20}), we deduce that
\begin{eqnarray*}
\lambda_2(B_\Pi^{(k,i)})-\lambda_2(B_\Pi^{(k,i+1)})
&=&(|T_{k}\cap \Gamma^{(i)}\cap \Gamma_{k+1}|-|T_{k}\cap \Gamma^{(i)}\cap \Gamma_{k+2,k+1}|)-\\
&&(|T_{k}\cap \Gamma^{(i+1)}\cap \Gamma_{k+1}|-|T_{k}\cap \Gamma^{(i+1)}\cap \Gamma_{k+2,k+1}|)\\
&=&(|T_{k}\cap \Gamma^{(i)}\cap \Gamma_{k+1}|-|T_{k}\cap \Gamma^{(i+1)}\cap \Gamma_{k+1}|)-\\
&&(|T_{k}\cap \Gamma^{(i)}\cap \Gamma_{k+2,k+1}|-|T_{k}\cap \Gamma^{(i+1)}\cap \Gamma_{k+2,k+1}|)\\
&=&(|T_k\cap \Gamma^{(i)}\cap \Gamma_{k+2}|-|T_k\cap \Gamma_{k+1}\cap \Gamma^{(i)}\cap \Gamma_{k+2}|)-\\
&&(|T_k\cap \Gamma^{(i)}\cap \Gamma_{k+3,k+2}|-|T_k\cap \Gamma_{k+1}\cap \Gamma^{(i)}\cap \Gamma_{k+3,k+2}|)\\
&=&|T_{k+1}\cap \Gamma^{(i)}\cap \Gamma_{k+2}|-|T_{k+1}\cap \Gamma^{(i)}\cap \Gamma_{k+3,k+2}|.
\end{eqnarray*}
Therefore, if $0\leq k\leq m-2$, we have $\lambda_2(B_\Pi^{(k,i)})-\lambda_2(B_\Pi^{(k,i+1)})=\lambda_2(B_\Pi^{(k+1,i)})$ again by  (\ref{eq-19}); if $k=m-1$, we have $\lambda_2(B_\Pi^{(m-1,i)})-\lambda_2(B_\Pi^{(m-1,i+1)})=|T_{m}\cap \Gamma^{(i)}\cap \Gamma_{m+1}|-|T_{m}\cap \Gamma^{(i)}\cap \Gamma_{m+2,m+1}|=|T_m|-0=|T_m|$ because $\mathrm{supp}(\tau)=\{1,2,\ldots,m\}$ for any $\tau\in T_m\cap \Gamma^{(i)}=T_m$.
\end{proof}

\begin{lem}\label{cong-lem}
Let $m$, $a$, $G_{k,i}$ and $H_{k,i}$ be defined as above. Assume that $a\geq 2$. For $0\leq i\leq a-2$ and $0\leq k\leq m-1$, we have $H_{k+1,i}\cong G_{k,i+1}$.
\end{lem}
\begin{proof}
According to  (\ref{eq-18}), we see that
$$H_{k+1,i}=\mathrm{Cay}(\Gamma^{(i)}\cap \Gamma_{k+1}, R_{k+1}\cap \Gamma^{(i)})=\mathrm{Cay}(\Gamma^{(i)}\cap \Gamma_{k+1}, T_k\cap  \Gamma_{k+1}\cap \Gamma^{(i)})$$
and
$$G_{k,i+1}=\mathrm{Cay}(\Gamma^{(i+1)}, T_k\cap \Gamma^{(i+1)}).$$
By the $(m+a)$-transitivity of $\Gamma$  acting on $[n]$, we can choose $\sigma\in \Gamma$ such that $\sigma$ moves $k+1$ to $n-i$ but fixes  $\{1,\ldots,k\}$ and $\{n-i+1,\ldots,n\}$ setwise. Then we see that  $\sigma^{-1}(\Gamma_{k+1}\cap \Gamma^{(i)})\sigma=\Gamma_{n-i}\cap\Gamma^{(i)}=\Gamma^{(i+1)}$ and  $\sigma^{-1}(T_k\cap  \Gamma_{k+1}\cap \Gamma^{(i)})\sigma=T_k\cap \Gamma_{n-i}\cap \Gamma^{(i)}=T_k\cap \Gamma^{(i+1)}$. Thus $\sigma$ induces an isomorphism from $H_{k+1,i}$ to $G_{k,i+1}$ naturally.
\end{proof}

Now we  give the main result of this section, which indicates that the problem of proving $\lambda_2(G_k)=\lambda_2(B_\Pi^{(k)})$ ($0\leq k\leq m-1$) can be reduced to verifying the result  for some small graphs.
\begin{thm}\label{normal-thm}
Let $\Gamma$ be a finite group acting on $[n]$, and let $G=\mathrm{Cay}(\Gamma,T)$ be a connected normal Cayley graph of $\Gamma$. Let $m=\max_{\tau\in T} |\mathrm{supp}(\tau)|$. If the action of $\Gamma$ on $[n]$ is $(m+a)$-transitive with $a\geq 1$ and $\lambda_2(G_{k,a-1})=\lambda_2(B_\Pi^{(k,a-1)})$ for all $k\in\{0,1,\ldots,m-1\}$, then we have
$$
\lambda_2(G_k)=\lambda_2(G_{k,0})=\lambda_2(B_\Pi^{(k,0)})=\lambda_2(B_\Pi^{(k)})=|T_{k}\cap \Gamma_{k+1}|-|T_{k}\cap \Gamma_{k+2,k+1}|,
$$
where $0\leq k\leq m-1$. In particular, $\lambda_2(G)=\lambda_2(G_0)=\lambda_2(B_\Pi^{(0)})=|T\cap \Gamma_1|-|T\cap \Gamma_{2,1}|$.
\end{thm}
\begin{proof}
If $a=1$, there is nothing to prove. Thus we assume that $a\geq 2$. The main idea is to prove $\lambda_2(G_{k,i})=\lambda_2(B_\Pi^{(k,i)})$ for all $0\leq k\leq m-1$ and $0\leq i\leq a-1$ by induction on $k$ and $i$.

First of all, we shall verify the induction basis. By assumption, we have known that $\lambda_2(G_{k,a-1})=\lambda_2(B_\Pi^{(k,a-1)})$ for all $0\leq k\leq m-1$. Thus it suffices to verify  $\lambda_2(G_{m-1,i})=\lambda_2(B_\Pi^{(m-1,i)})$ for all  $0\leq i\leq a-1$. If $i=a-1$, we obtain the result again by assumption. Now suppose $0\leq i<a-1$, and assume that the result holds for $i+1$, i.e., $\lambda_2(G_{m-1,i+1})=\lambda_2(B_\Pi^{(m-1,i+1)})$. We shall prove $\lambda_2(G_{m-1,i})=\lambda_2(B_\Pi^{(m-1,i)})$. According to the arguments below Theorem \ref{Cay-thm} and (\ref{eq-8}),  we only need to show $\lambda_2(B_\Pi^{(m-1,i)})\geq \lambda_2(H_{m,i})+\lambda_2(G_{m,i})$. From Lemma \ref{cong-lem} we see that $H_{m,i}\cong G_{m-1,i+1}$,  so $\lambda_2(H_{m,i})=\lambda_2(G_{m-1,i+1})=\lambda_2(B_\Pi^{(m-1,i+1)})$ by the induction hypothesis. Also, as mentioned above, we have $\lambda_2(G_{m,i})=|T_m\cap \Gamma^{(i)}|=|T_m|$ because $G_{m,i}$ is disconnected. Therefore, from Lemma \ref{recursive-lem} we deduce that
$$
\lambda_2(B_\Pi^{(m-1,i)})-\lambda_2(H_{m,i})=\lambda_2(B_\Pi^{(m-1,i)})-\lambda_2(B_\Pi^{(m-1,i+1)})=|T_m|=\lambda_2(G_{m,i}),
$$
as required. Thus we have built up  the induction basis.

Now suppose $0\leq k<m-1$ and $0\leq i<a-1$, and assume that the result holds for $k+1,i$ and $k,i+1$, i.e., $\lambda_2(G_{k+1,i})=\lambda_2(B_\Pi^{(k+1,i)})$ and $\lambda_2(G_{k,i+1})=\lambda_2(B_\Pi^{(k,i+1)})$. We shall prove $\lambda_2(G_{k,i})=\lambda_2(B_\Pi^{(k,i)})$. As above, it remains to show that  $\lambda_2(B_\Pi^{(k,i)})\geq \lambda_2(H_{k+1,i})+\lambda_2(G_{k+1,i})$. Again by Lemma \ref{cong-lem} and  the induction hypothesis, we have $\lambda_2(H_{k+1,i})=\lambda_2(G_{k,i+1})=\lambda_2(B_\Pi^{(k,i+1)})$ and $\lambda_2(G_{k+1,i})=\lambda_2(B_\Pi^{(k+1,i)})$. Then from Lemma \ref{recursive-lem} we obtain
$$
\lambda_2(B_\Pi^{(k,i)})-\lambda_2(H_{k+1,i})=\lambda_2(B_\Pi^{(k,i)})-\lambda_2(B_\Pi^{(k,i+1)})=\lambda_2(B_\Pi^{(k+1,i)})=\lambda_2(G_{k+1,i}),
$$
and the result follows.

Therefore, we may conclude that $\lambda_2(G_{k,i})=\lambda_2(B_\Pi^{(k,i)})$ for all $0\leq k\leq m-1$ and $0\leq i\leq a-1$. In particular, for $0\leq k\leq m-1$, we have
$
\lambda_2(G_k)=\lambda_2(G_{k,0})=\lambda_2(B_\Pi^{(k,0)})=|T_{k}\cap \Gamma_{k+1}|-|T_{k}\cap \Gamma_{k+2,k+1}|.
$
\end{proof}

According to Theorem \ref{normal-thm}, to prove $\lambda_2(G)=\lambda_2(G_0)=\lambda_2(B_\Pi^{(0)})=|T\cap \Gamma_1|-|T\cap \Gamma_{2,1}|$ (and as by-products, $\lambda_2(G_k)=\lambda_2(B_\Pi^{(k)})$ for $1\leq k\leq m-1$), it suffices to verify $\lambda_2(G_{k,a-1})=\lambda_2(B_\Pi^{(k,a-1)})$ for all $k\in\{0,1,\ldots,m-1\}$. Note that if $a$ is relatively large, i.e., the action of $\Gamma$ on $[n]$ is of high transitivity, then the graph $G_{k,a-1}$ will be of small order.
This makes it  easier to verify the equalities. It is well known that the symmetric group $S_n$ acts $n$-transitively on $[n]$, so Theorem \ref{normal-thm} is particularly effective for normal Cayley graphs of $S_n$. In the next section, we consider to determine the second eigenvalues of connected normal Cayley graphs of $S_n$ with $m\leq 5$.

\section{The second eigenvalues of normal Cayley graphs of symmetric groups}\label{s-4}

Let $\Gamma=S_n$ be the symmetric group on $[n]$ with $n\geq 3$. It is well known that $S_n$ acts $n$-transitively on $[n]$, and that two elements in $S_n$ are conjugated if and only if they share the same cycle type. Let $G=\mathrm{Cay}(S_n,T)$ be a normal Cayley graph of $S_n$, that is, $T$ is the disjoint union of some conjugacy classes of $S_n$. Then $G$ is connected if and only if $T$ contains some odd permutation. This is because $T$ generates a non-identity normal subgroup  of $S_n$ while $A_n$ is the unique nontrivial normal subgroup of $S_n$ for $n\neq 4$, and $A_4$ and $\{e,(1,2)(3,4),(1,3)(2,4),(1,4)(2,3)\}\leq A_4$ are the only nontrivial normal subgroups of $S_n$ for $n=4$.

In this section, as applications of Theorem  \ref{normal-thm}, we consider  the second eigenvalues of connected normal Cayley graphs of $S_n$ for which each element of the connection set has at most five supports.

\renewcommand\arraystretch{1.2}
\begin{table}[p]
\caption{The  structure of $\mathcal{C}_k^{(i)}$ for $1\leq i\leq 6$ and $k\in [n]$.\label{tab-1}}
\scriptsize
\begin{center}
\begin{tabular*}{\textwidth}{l@{\extracolsep{\fill}}l@{\extracolsep{\fill}}l}
\hline\noalign{\smallskip}
$i$&$k$&$\mathcal{C}_k^{(i)}$\\
\noalign{\smallskip}\hline\noalign{\smallskip}
$1$&$1$&$\{(1,q)\mid 2\leq q\leq n\}$   \\
$1$&$2$&$\{(1,2)\}$   \\
$1$&$\geq 3$&$\emptyset$   \\
$2$&$1$&$\{(1,q,r)\mid 2\leq q,r\leq n\}$   \\
$2$&$2$&$\{(1,2,r),(1,r,2)\mid 3\leq r\leq n\}$   \\
$2$&$3$&$\{(1,2,3),(1,3,2)\}$   \\
$2$&$\geq 4$&$\emptyset$   \\
$3$&$1$&$\{(1,q)(r,s)\mid 2\leq q,r,s\leq n\}$   \\
$3$&$2$&$\{(1,2)(r,s),(1,r)(2,s)\mid 3\leq r,s\leq n\}$   \\
$3$&$3$&$\{(1,2)(3,s),(1,3)(2,s),(1,s)(2,3)\mid 4\leq s\leq n\}$   \\
$3$&$4$&$\{(1,2)(3,4),(1,3)(2,4),(1,4)(2,3)\}$   \\
$3$&$\geq5$&$\emptyset$   \\
$4$&$1$&$\{(1,q,r,s)\mid 2\leq q,r,s\leq n\}$   \\
$4$&$2$&$\{(1,2,r,s),(1,r,2,s),(1,r,s,2)\mid 3\leq r,s\leq n\}$   \\
$4$&$3$&$\{(1,2,3,s),(1,2,s,3),(1,3,2,s),(1,3,s,2),(1,s,2,3),(1,s,3,2)\mid 4\leq s\leq n\}$   \\
$4$&$4$&$\{(1,2,3,4),(1,2,4,3),(1,3,2,4),(1,3,4,2),(1,4,2,3),(1,4,3,2)\}$   \\
$4$&$\geq5$&$\emptyset$ \\
$5$&$1$&$\{(1,p,q)(r,s),(p,q,r)(1,s)\mid 2\leq p,q,r,s\leq n\}$   \\
$5$&$2$&$\{(p,q,r)(1,2), (1,p,q)(2,r), (2,p,q)(1,r),(1,2,p)(q,r),(1,p,2)(q,r)\mid 3\leq p,q,r\leq n\}$   \\
$5$&$3$&$\left\{
\begin{array}{l|}
(1,2,3)(p,q),(1,3,2)(p,q),(1,2,p)(3,q),(1,p,2)(3,q),(1,3,p)(2,q),(1,p,3)(2,q),\\
(2,3,p)(1,q),(2,p,3)(1,q),(1,p,q)(2,3),(2,p,q)(1,3),(3,p,q)(1,2)
\end{array}
~\mbox{$4\leq p,q\leq n$}
\right\}$   \\
$5$&$4$&$\left\{
\begin{array}{l|}
(1,2,3)(4,p),(1,3,2)(4,p),(1,2,4)(3,p),(1,4,2)(3,p),(1,2,p)(3,4),\\
(1,p,2)(3,4),(1,3,4)(2,p),(1,4,3)(2,p),(1,3,p)(2,4),(1,p,3)(2,4),\\
(1,4,p)(2,3),(1,p,4)(2,3),(2,3,4)(1,p),(2,4,3)(1,p),(2,3,p)(1,4),\\
(2,p,3)(1,4),(2,4,p)(1,3),(2,p,4)(1,3),(3,4,p)(1,2),(3,p,4)(1,2)
\end{array}
~\mbox{$5\leq p\leq n$}
\right\}$   \\
$5$&$5$&$\left\{
\begin{array}{l}
(1,2,3)(4,5),(1,3,2)(4,5),(1,2,4)(3,5),(1,4,2)(3,5),(1,2,5)(3,4),\\
(1,5,2)(3,4),(1,3,4)(2,5),(1,4,3)(2,5),(1,3,5)(2,4),(1,5,3)(2,4),\\
(1,4,5)(2,3),(1,5,4)(2,3),(2,3,4)(1,5),(2,4,3)(1,5),(2,3,5)(1,4),\\
(2,5,3)(1,4),(2,4,5)(1,3),(2,5,4)(1,3),(3,4,5)(1,2),(3,5,4)(1,2)
\end{array}\right\}$    \\
$5$&$\geq6$&$\emptyset$ \\
$6$&$1$&$\{(1,q,r,s,t)\mid 2\leq q,r,s,t\leq n\}$   \\
$6$&$2$&$\{(1,2,r,s,t),(1,r,2,s,t),(1,r,s,2,t),(1,r,s,t,2)\mid 3\leq r,s,t\leq n\}$   \\
$6$&$3$&$\left\{
\begin{array}{l|}
(1,2,3,s,t),(1,3,2,s,t),(1,2,s,3,t),(1,3,s,2,t),(1,2,s,t,3),(1,3,s,t,2),\\
(1,s,2,3,t),(1,s,3,2,t),(1,s,2,t,3),(1,s,3,t,2),(1,s,t,2,3),(1,s,t,3,2)
\end{array}
~\mbox{$4\leq s,t\leq n$}
\right\}$\\
$6$&$4$&$\left\{
\begin{array}{l|}
(1,2,3,4,t),(1,2,3,t,4),(1,2,4,3,t),(1,2,4,t,3),(1,2,t,3,4),(1,2,t,4,3),\\
(1,3,2,4,t),(1,3,2,t,4),(1,3,4,2,t),(1,3,4,t,2),(1,3,t,2,4),(1,3,t,4,2),\\
(1,4,2,3,t),(1,4,2,t,3),(1,4,3,2,t),(1,4,3,t,2),(1,4,t,2,3),(1,4,t,3,2),\\
(1,t,2,3,4),(1,t,2,4,3),(1,t,3,2,4),(1,t,3,4,2),(1,t,4,2,3),(1,t,4,3,2)
\end{array}
~\mbox{$5\leq t\leq n$}
\right\}$   \\
$6$&$5$&$\left\{
\begin{array}{l}
(1,2,3,4,5),(1,2,3,5,4),(1,2,4,3,5),(1,2,4,5,3),(1,2,5,3,4),(1,2,5,4,3),\\
(1,3,2,4,5),(1,3,2,5,4),(1,3,4,2,5),(1,3,4,5,2),(1,3,5,2,4),(1,3,5,4,2),\\
(1,4,2,3,5),(1,4,2,5,3),(1,4,3,2,5),(1,4,3,5,2),(1,4,5,2,3),(1,4,5,3,2),\\
(1,5,2,3,4),(1,5,2,4,3),(1,5,3,2,4),(1,5,3,4,2),(1,5,4,2,3),(1,5,4,3,2)
\end{array}
\right\}$    \\
$6$&$\geq6$&$\emptyset$ \\
\hline
\end{tabular*}
\end{center}
\end{table}

For convenience, we first list all the nontrivial conjugacy classes of $S_n$ with each element having at most five supports:
\begin{equation}\label{eq-21}
\left\{
\begin{aligned}
&\mathcal{C}^{(1)}=\{(p,q)\mid 1\leq p,q\leq n\};\\
&\mathcal{C}^{(2)}=\{(p,q,r)\mid 1\leq p,q,r\leq n\};\\
&\mathcal{C}^{(3)}=\{(p,q)(r,s)\mid 1\leq p,q,r,s\leq n\};\\
&\mathcal{C}^{(4)}=\{(p,q,r,s)\mid 1\leq p,q,r,s\leq n\};\\
&\mathcal{C}^{(5)}=\{(p,q,r)(s,t)\mid 1\leq p,q,r,s,t\leq n\};\\
&\mathcal{C}^{(6)}=\{(p,q,r,s,t)\mid 1\leq p,q,r,s,t\leq n\},
\end{aligned}
\right.
\end{equation}
where $p,q,r,s,t$ are pairwise  distinct. For  $k\in[n]$, we denote by $\mathcal{C}_k^{(i)}$ (see Table \ref{tab-1}) the set of elements in $\mathcal{C}^{(i)}$ that moves each point of $\{1,2,\ldots,k\}$, where $1\leq i\leq 6$.

Now suppose that  $G=\mathrm{Cay}(S_n,T)$ ($=G_0$)  is a  normal Cayley graph of $S_n$ with $m=\max_{\tau\in T} |\mathrm{supp}(\tau)|\leq 5$. For $k\in [n]$, let $T_k=T\setminus(T\cap (\cup_{i=1}^k (S_n)_{i}))$ (see Claim \ref{claim-2})  and  $G_k=\mathrm{Cay}(S_n,T_k)$  be defined  as in  (\ref{eq-13}).  Then $T$ ($=T_0$)  and $T_k$ ($k\in [n]$) can be respectively written as $T=\cup_{i\in \mathcal{I}_T}\mathcal{C}^{(i)}$ (see (\ref{eq-21})) and $T_k=\cup_{i\in \mathcal{I}_T}\mathcal{C}_k^{(i)}$ (see Table \ref{tab-1}), where $\mathcal{I}_T$ is  some nonempty subset of $\{1,2,3,4,5,6\}$. Moreover, by the arguments at the beginning of this section, we obtain that $G=\mathrm{Cay}(S_n,T)$ is connected if and only if  $T=\cup_{i\in \mathcal{I}_T}\mathcal{C}^{(i)}$ with
\begin{equation}\label{eq-23}
\mathcal{I}_T\in \mathcal{P}\setminus\{\emptyset,\{2\},\{3\},\{6\},\{2,3\},\{2,6\},\{3,6\},\{2,3,6\}\}
\end{equation}
where $\mathcal{P}$ is the power set of $\{1,2,\ldots,6\}$.

Now we give the main result of this section, which determines the second eigenvalues of a majority of connected normal  Cayley graphs (and some subgraphs of these graphs) on $S_n$ satisfying $m=\max_{\tau\in T}|\mathrm{supp}(\tau)|\leq 5$.

\begin{thm}\label{symmetric-thm}
Let $G=\mathrm{Cay}(S_n,T)$ ($=G_0$)  be a connected normal Cayley graph of $S_n$ ($n\geq 7$) with $m=\max_{\tau\in T}|\mathrm{supp}(\tau)|\leq 5$ (that is, $T=\cup_{i\in \mathcal{I}_T}\mathcal{C}^{(i)}$ with $\mathcal{I}_T$ given in (\ref{eq-23})). Let $G_k$ and $T_k$ be defined as in  (\ref{eq-13}). If $\mathcal{I}_T\neq \{1,3\},\{1,6\},\{4,6\},\{1,2,3\}, $ $\{1,2,6\},\{1,3,6\},\{1,4,6\},\{2,4,6\},\{3,4,6\},\{1,2,3,6\},\{1,2,4,6\}, \{1,3,4,6\}, \{2,3,$\\ $4,6\}, \{2,3,5,6\}, \{1,2,3,4,6\}$, then for $0\leq k\leq m-1$,  the graph $G_k$ is connected and  has second eigenvalue
$$\lambda_2(G_k)=\lambda_2(B_\Pi^{(k)})=|T_k\cap (S_n)_{k+1}|-|T_k\cap (S_n)_{k+2,k+1}|.$$
\end{thm}
\begin{proof}
Take  $a=n-6$ ($\geq 1$). Since $n\geq 7$ and $m\leq 5$,  we see that $S_n$ acts $(m+a)$-transitively on $[n]$ due to $m+a<n$.  By Theorem \ref{normal-thm}, to prove $\lambda_2(G_k)=\lambda_2(B_\Pi^{(k)})$ for $0\leq k\leq m-1$, it remains to verify $\lambda_2(G_{k,a-1})=\lambda_2(B_\Pi^{(k,a-1)})$  for $0\leq k\leq m-1$. Since $S_n^{(a-1)}=S_n^{(n-7)}=\cap_{i=1}^{n-7}(S_n)_{n-i+1}\cong S_7$, we have $G_{k,a-1}=\mathrm{Cay}(S_n^{(n-7)},T_k\cap S_n^{(n-7)})\cong\mathrm{Cay}(S_7,T_k\cap S_7)$ according to  (\ref{eq-18}). Also note that
$\lambda_2(B_\Pi^{(k,a-1)})=|T_{k}\cap S_n^{(a-1)}\cap (S_n)_{k+1}|-|T_{k}\cap S_n^{(a-1)}\cap (S_n)_{k+2,k+1}|=|T_{k} \cap (S_7)_{k+1}|-|T_{k}\cap (S_7)_{k+2,k+1}|$ by (\ref{eq-19}). Thus the problem is reduced to verify
\begin{equation}\label{eq-24}
\lambda_2(\mathrm{Cay}(S_7,T_k\cap S_7))=|T_{k} \cap (S_7)_{k+1}|-|T_{k}\cap (S_7)_{k+2,k+1}|
\end{equation}
for $0\leq k\leq m-1$. Recall that $T_0=T=\cup_{i\in \mathcal{I}_T}\mathcal{C}^{(i)}$ with $\mathcal{I}_T$ given in (\ref{eq-23}), and $T_k=\cup_{i\in \mathcal{I}_T}\mathcal{C}_k^{(i)}$ is just the set of $\tau\in T$ such that $\{1,2,\ldots,k\}\subseteq \mathrm{supp}(\tau)$ for $1\leq k\leq m-1$. Using computer, we can check that  (\ref{eq-24}) is true except for those $T$'s with $\mathcal{I}_T=\{1,3\},\{1,6\},\{4,6\},\{1,2,3\}, \{1,2,6\},\{1,3,6\},\{1,4,6\},\{2,4,6\},\{3,4,6\},\{1,2,3,6\},$ $\{1,2,4,6\}, \{1,3,4,6\}, \{2,3,4,6\}, \{2,3,5,6\}$ or $\{1,2,3,4,6\}$. Therefore, for the remaining $T$'s,  we may conclude that
$$\lambda_2(G_k)=\lambda_2(B_\Pi^{(k)})=|T_k\cap (S_n)_{k+1}|-|T_k\cap (S_n)_{k+2,k+1}|,$$
where $0\leq k\leq m-1$ (in Table \ref{tab-2}, we list the exact values of the first two largest eigenvalues of these $G_k$'s); and furthermore, we observe that $\lambda_2(G_k)=\lambda_2(B_\Pi^{(k)})<|T_k|=\lambda_1(G_k)$, so  $G_k$ is also connected for $1\leq k\leq m-1$. 

We complete the proof.
\end{proof}

\begin{center}
\scriptsize
\renewcommand\arraystretch{1.0}
\begin{longtable}{lllll}

\caption{\small The  first two  eigenvalues of $G_k=\mathrm{Cay}(S_n,T_k)$, where $T_k=\cup_{i\in \mathcal{I}_T}\mathcal{C}_k^{(i)}$.} \label{tab-2} \\
\hline\noalign{\smallskip}
\multicolumn{1}{l}{$\mathcal{I}_T$} &\multicolumn{1}{l}{$m$} &\multicolumn{1}{l}{$k$} & \multicolumn{1}{l}{$\lambda_1(G_k)$} & \multicolumn{1}{l}{$\lambda_2(G_k)$}\\ 
\noalign{\smallskip}\hline\noalign{\smallskip}
\endfirsthead

\multicolumn{5}{l}
{{continued from previous page}} \\
\hline\noalign{\smallskip}
\multicolumn{1}{l}{$\mathcal{I}_T$} &\multicolumn{1}{l}{$m$} &\multicolumn{1}{l}{$k$} & \multicolumn{1}{l}{$\lambda_1(G_k)$} & \multicolumn{1}{l}{$\lambda_2(G_k)$}\\ 
\noalign{\smallskip}\hline\noalign{\smallskip}
\endhead

\hline \multicolumn{5}{r}{{continued on next page}} \\ 
\endfoot
\hline 
\endlastfoot

$\{1\}$& $2$ & $0$ & $(n(n\!-\!1))/2$ & $(n(n\!-\!3))/2$\\
& & $1$ & $n\!-\!1$ & $n\!-\!2$\\
$\{4\}$& $4$  & $0$ & $(n(n\!-\!1)(n\!-\!2)(n\!-\!3))/4$ & $(n(n\!-\!2)(n\!-\!3)(n\!-\!5))/4$\\
& & $1$ & $(n\!-\!1)(n\!-\!2)(n\!-\!3)$ & $(n\!-\!3)(n^2\!-\!6n\!+\!6)$\\
& & $2$ & $3(n\!-\!2)(n\!-\!3)$ & $3n^2\!-\!21n\!+\!34$\\
& & $3$ & $6(n\!-\!3)$ & $6(n\!-\!4)$\\
$\{5\}$& $5$  & $0$ & $(n(n\!-\!1)(n\!-\!2)(n\!-\!3)(n\!-\!4))/6$ & $(n(n\!-\!2)(n\!-\!3)(n\!-\!4)(n\!-\!6))/6$\\
& & $1$ & $(5(n\!-\!1)(n\!-\!2)(n\!-\!3)(n\!-\!4))/6$ & $(5(n\!-\!3)(n\!-\!4)(n^2\!-\!7n\!+\!7))/6$\\
& & $2$ & $(10(n\!-\!2)(n\!-\!3)(n\!-\!4))/3$ & $(5(n\!-\!4)(2n^2\!-\!16n\!+\!27))/3$\\
& & $3$ & $10(n\!-\!3)(n\!-\!4)$ & $5(2n^2\!-\!18n\!+\!39)$\\
& & $4$ & $20(n\!-\!4)$ & $20(n\!-\!5)$\\
$\{1,2\}$& $3$  & $0$ & $(n(2n\!-\!1)(n\!-\!1))/6$ & $(n(n\!-\!1)(2n\!-\!7))/6$\\
& & $1$ & $(n\!-\!1)^2$ & $(n\!-\!1)(n\!-\!3)$\\
& & $2$ & $2n\!-\!3$ & $2n\!-\!5$\\
$\{1,4\}$& $4$ & $0$ & $(n(n\!-\!1)(n^2\!-\!5n\!+\!8))/4$ & $(n(n\!-\!4)(n\!-\!3)^2)/4$\\
& & $1$ & $(n\!-\!1)(n^2\!-\!5n\!+\!7)$ & $(n\!-\!4)(n^2\!-\!5n\!+\!5)$\\
& & $2$ & $3n^2\!-\!15n\!+\!19$ & $3n^2\!-\!21n\!+\!35$\\
& & $3$ & $6(n\!-\!3)$ & $6(n\!-\!4)$\\
$\{1,5\}$& $5$ & $0$ & $(n(n\!-\!1)(n^3\!-\!9n^2\!+\!26n\!-\!21))/6$ & $(n(n\!-\!5)(n\!-\!3)(n^2\!-\!7n\!+\!9))/6$\\
& & $1$ & $((n\!-\!1)(5n^3\!-\!45n^2\!+\!130n\!-\!114))/6$ & $(5n^4\!-\!70n^3\!+\!340n^2\!-\!659n\!+\!408)/6$\\
& & $2$ & $(10n^3\!-\!90n^2\!+\!260n\!-\!237)/3$ & $(10n^3\!-\!120n^2\!+\!455n\!-\!537)/3$\\
& & $3$ & $10(n\!-\!3)(n\!-\!4)$ & $5(2n^2\!-\!18n\!+\!39)$\\
& & $4$ & $20(n\!-\!4)$ & $20(n\!-\!5)$\\
$\{2,4\}$& $4$ & $0$ & $(n(3n\!-\!5)(n\!-\!1)(n\!-\!2))/12$ & $(n(n\!-\!2)(3n^2\!-\!20n\!+\!29))/12$\\
& & $1$ & $(n\!-\!1)(n\!-\!2)^2$ & $n^3\!-\!8n^2\!+\!19n\!-\!13$\\
& & $2$ & $(n\!-\!2)(3n\!-\!7)$ & $(3n\!-\!7)(n\!-\!4)$\\
& & $3$ & $2(3n\!-\!8)$ & $2(3n\!-\!11)$\\
$\{2,5\}$& $5$ & $0$ & $(n(n\!-\!1)(n\!-\!2)(n^2\!-\!7n\!+\!14))/6$ & $(n(n\!-\!2)(n\!-\!5)(n\!-\!4)^2)/6$\\
& & $1$ & $((n\!-\!1)(n\!-\!2)(5n^2\!-\!35n\!+\!66))/6$ & $((n\!-\!5)(5n^3\!-\!45n^2\!+\!121n\!-\!90))/6$\\
& & $2$ & $(2(n\!-\!2)(5n^2\!-\!35n\!+\!63))/3$ & $(10n^3\!-\!120n^2\!+\!461n\!-\!558)/3$\\
& & $3$ & $2(5n^2\!-\!35n\!+\!61)$ & $10n^2\!-\!90n\!+\!197$\\
& & $4$ & $20(n\!-\!4)$ & $20(n\!-\!5)$\\
$\{3,4\}$& $4$ & $0$ & $(3n(n\!-\!1)(n\!-\!2)(n\!-\!3))/8$ & $(3n(n\!-\!2)(n\!-\!3)(n\!-\!5))/8$\\
& & $1$ & $(3(n\!-\!1)(n\!-\!2)(n\!-\!3))/2$ & $(3(n\!-\!3)(n^2\!-\!6n\!+\!6))/2$\\
& & $2$ & $(9(n\!-\!2)(n\!-\!3))/2$ & $(3(3n^2\!-\!21n\!+\!34))/2$\\
& & $3$ & $9(n\!-\!3)$ & $9(n\!-\!4)$\\
$\{3,5\}$& $5$ & $0$ & $(n(n\!-\!1)(n\!-\!2)(n\!-\!3)(4n\!-\!13))/24$ & $(n(n\!-\!2)(n\!-\!3)(4n^2\!-\!37n\!+\!81))/24$\\
& & $1$ & $((n\!-\!1)(n\!-\!2)(n\!-\!3)(5n\!-\!17))/6$ & $((n\!-\!3)(5n^3\!-\!52n^2\!+\!157n\!-\!122))/6$\\
& & $2$ & $((n\!-\!3)(20n\!-\!71)(n\!-\!2))/6$ & $(20n^3\!-\!231n^2\!+\!847n\!-\!978)/6$\\
& & $3$ & $(n\!-\!3)(10n\!-\!37)$ & $10n^2\!-\!87n\!+\!183$\\
& & $4$ & $20n\!-\!77$ & $20n\!-\!97$\\
$\{4,5\}$& $5$ & $0$ & $(n(2n\!-\!5)(n\!-\!1)(n\!-\!2)(n\!-\!3))/12$ & $(n(n\!-\!2)(2n\!-\!11)(n\!-\!3)^2)/12$\\
& & $1$ & $((5n\!-\!14)(n\!-\!1)(n\!-\!2)(n\!-\!3))/6$ & $((n\!-\!3)(5n^3\!-\!49n^2\!+\!139n\!-\!104))/6$\\
& & $2$ & $((n\!-\!3)(10n\!-\!31)(n\!-\!2))/3$ & $(10n^3\!-\!111n^2\!+\!392n\!-\!438)/3$\\
& & $3$ & $2(n\!-\!3)(5n\!-\!17)$ & $10n^2\!-\!84n\!+\!171$\\
& & $4$ & $2(10n\!-\!37)$ & $2(10n\!-\!47)$\\
$\{5,6\}$& $5$ & $0$ & $(11n(n\!-\!1)(n\!-\!2)(n\!-\!3)(n\!-\!4))/30$ & $(11n(n\!-\!2)(n\!-\!3)(n\!-\!4)(n\!-\!6))/30$\\
& & $1$ & $(11(n\!-\!1)(n\!-\!2)(n\!-\!3)(n\!-\!4))/6$ & $(11(n\!-\!4)(n\!-\!3)(n^2\!-\!7n\!+\!7))/6$\\
& & $2$ & $(22(n\!-\!3)(n\!-\!4)(n\!-\!2))/3$ & $(11(n\!-\!4)(2n^2\!-\!16n\!+\!27))/3$\\
& & $3$ & $22(n\!-\!3)(n\!-\!4)$ & $11(2n^2\!-\!18n\!+\!39)$\\
& & $4$ & $44(n\!-\!4)$ & $44(n\!-\!5)$\\
$\{1,2,4\}$& $4$ & $0$ & $(n(n\!-\!1)(3n^2\!-\!11n\!+\!16))/12$ & $(n(n\!-\!4)(3n^2\!-\!14n\!+\!19))/12$\\
& & $1$ & $(n\!-\!1)(n^2\!-\!4n\!+\!5)$ & $(n\!-\!3)(n^2\!-\!5n\!+\!5)$\\
& & $2$ & $3n^2\!-\!13n\!+\!15$ & $3n^2\!-\!19n\!+\!29$\\
& & $3$ & $2(3n\!-\!8)$ & $2(3n\!-\!11)$\\
$\{1,2,5\}$& $5$ & $0$ & $(n(n\!-\!1)(n^3\!-\!9n^2\!+\!28n\!-\!25))/6$ & $(n(n^4\!-\!15n^3\!+\!82n^2\!-\!189n\!+\!151))/6$\\
& & $1$ & $((n\!-\!1)(5n^3\!-\!45n^2\!+\!136n\!-\!126))/6$ & $((n\!-\!3)(5n^3\!-\!55n^2\!+\!181n\!-\!146))/6$\\
& & $2$ & $(10n^3\!-\!90n^2\!+\!266n\!-\!249)/3$ & $((n\!-\!5)(10n^2\!-\!70n\!+\!111))/3$\\
& & $3$ & $2(5n^2\!-\!35n\!+\!61)$ & $10n^2\!-\!90n\!+\!197$\\
& & $4$ & $20(n\!-\!4)$ & $20(n\!-\!5)$\\
$\{1,3,4\}$& $4$ & $0$ & $(n(n\!-\!1)(3n^2\!-\!15n\!+\!22))/8$ & $(n(n\!-\!3)(3n^2\!-\!21n\!+\!34))/8$\\
& & $1$ & $((n\!-\!1)(3n^2\!-\!15n\!+\!20))/2$ & $(3n^3\!-\!27n^2\!+\!74n\!-\!58)/2$\\
& & $2$ & $((3n\!-\!7)(3n\!-\!8))/2$ & $((3n\!-\!8)(3n\!-\!13))/2$\\
& & $3$ & $9(n\!-\!3)$ & $9(n\!-\!4)$\\
$\{1,3,5\}$& $5$ & $0$ & $(n(n\!-\!1)(4n^3\!-\!33n^2\!+\!89n\!-\!66))/24$ & $(n(n\!-\!3)(n\!-\!5)(4n^2\!-\!25n\!+\!30))/24$\\
& & $1$ & $((n\!-\!1)(5n^3\!-\!42n^2\!+\!115n\!-\!96))/6$ & $(5n^4\!-\!67n^3\!+\!313n^2\!-\!587n\!+\!354)/6$\\
& & $2$ & $(20n^3\!-\!171n^2\!+\!475n\!-\!420)/6$ & $((n\!-\!4)(20n^2\!-\!151n\!+\!243))/6$\\
& & $3$ & $(n\!-\!3)(10n\!-\!37)$ & $10n^2\!-\!87n\!+\!183$\\
& & $4$ & $20n\!-\!77$ & $20n\!-\!97$\\
$\{1,4,5\}$& $5$ & $0$ & $(n(2n^2\!-\!13n\!+\!24)(n\!-\!1)^2)/12$ & $(n(n\!-\!3)(n\!-\!4)(n\!-\!5)(2n\!-\!3))/12$\\
& & $1$ & $((n\!-\!1)(5n^3\!-\!39n^2\!+\!100n\!-\!78))/6$ & $((n\!-\!4)(n\!-\!5)(5n^2\!-\!19n\!+\!15))/6$\\
& & $2$ & $(10n^3\!-\!81n^2\!+\!215n\!-\!183)/3$ & $((n\!-\!5)(10n^2\!-\!61n\!+\!87))/3$\\
& & $3$ & $2(n\!-\!3)(5n\!-\!17)$ & $10n^2\!-\!84n\!+\!171$\\
& & $4$ & $2(10n\!-\!37)$ & $2(10n\!-\!47)$\\
$\{1,5,6\}$& $5$ & $0$ & $(n(n\!-\!1)(11n^3\!-\!99n^2\!+\!286n\!-\!249))/30$ & $(n(n\!-\!3)(11n^3\!-\!132n^2\!+\!484n\!-\!513))/30$\\
& & $1$ & $((n\!-\!1)(11n^3\!-\!99n^2\!+\!286n\!-\!258))/6$ & $(11n^4\!-\!154n^3\!+\!748n^2\!-\!1457n\!+\!912)/6$\\
& & $2$ & $(22n^3\!-\!198n^2\!+\!572n\!-\!525)/3$ & $(22n^3\!-\!264n^2\!+\!1001n\!-\!1185)/3$\\
& & $3$ & $22(n\!-\!3)(n\!-\!4)$ & $11(2n^2\!-\!18n\!+\!39)$\\
& & $4$ & $44(n\!-\!4)$ & $44(n\!-\!5)$\\
$\{2,3,4\}$& $4$ & $0$ & $(n(n\!-\!1)(n\!-\!2)(9n\!-\!19))/24$ & $(n(n\!-\!2)(9n^2\!-\!64n\!+\!103))/24$\\
& & $1$ & $((n\!-\!1)(n\!-\!2)(3n\!-\!7))/2$ & $(3n^3\!-\!25n^2\!+\!62n\!-\!44)/2$\\
& & $2$ & $((n\!-\!2)(9n\!-\!23))/2$ & $(9n^2\!-\!59n\!+\!90)/2$\\
& & $3$ & $9n\!-\!25$ & $9n\!-\!34$\\
$\{2,3,5\}$& $5$ & $0$ & $(n(n\!-\!1)(n\!-\!2)(4n^2\!-\!25n\!+\!47))/24$ & $(n(n\!-\!2)(n\!-\!5)(4n^2\!-\!29n\!+\!55))/24$\\
& & $1$ & $((n\!-\!1)(n\!-\!2)(5n^2\!-\!32n\!+\!57))/6$ & $((n\!-\!4)(5n^3\!-\!47n^2\!+\!131n\!-\!99))/6$\\
& & $2$ & $((n\!-\!2)(20n^2\!-\!131n\!+\!225))/6$ & $(20n^3\!-\!231n^2\!+\!859n\!-\!1014)/6$\\
& & $3$ & $10n^2\!-\!67n\!+\!113$ & $(10n\!-\!37)(n\!-\!5)$\\
& & $4$ & $20n\!-\!77$ & $20n\!-\!97$\\
$\{2,4,5\}$& $5$ & $0$ & $(n(n\!-\!1)(n\!-\!2)(2n^2\!-\!11n\!+\!19))/12$ & $(n(n\!-\!2)(n\!-\!5)(2n^2\!-\!13n\!+\!23))/12$\\
& & $1$ & $((n\!-\!1)(n\!-\!2)(5n^2\!-\!29n\!+\!48))/6$ & $(5n^4\!-\!64n^3\!+\!292n^2\!-\!551n\!+\!342)/6$\\
& & $2$ & $((n\!-\!2)(10n^2\!-\!61n\!+\!99))/3$ & $((n\!-\!4)(10n^2\!-\!71n\!+\!114))/3$\\
& & $3$ & $2(5n^2\!-\!32n\!+\!52)$ & $10n^2\!-\!84n\!+\!173$\\
& & $4$ & $2(10n\!-\!37)$ & $2(10n\!-\!47)$\\
$\{2,5,6\}$& $5$ & $0$ & $(n(n\!-\!1)(n\!-\!2)(11n^2\!-\!77n\!+\!142))/30$ & $(n(n\!-\!2)(n\!-\!4)(11n^2\!-\!99n\!+\!208))/30$\\
& & $1$ & $((n\!-\!1)(n\!-\!2)(11n^2\!-\!77n\!+\!138))/6$ & $(11n^4\!-\!154n^3\!+\!754n^2\!-\!1493n\!+\!954)/6$\\
& & $2$ & $(2(n\!-\!2)(11n^2\!-\!77n\!+\!135))/3$ & $(22n^3\!-\!264n^2\!+\!1007n\!-\!1206)/3$\\
& & $3$ & $2(11n^2\!-\!77n\!+\!133)$ & $22n^2\!-\!198n\!+\!431$\\
& & $4$ & $44(n\!-\!4)$ & $44(n\!-\!5)$\\
$\{3,4,5\}$& $5$ & $0$ & $(n(4n\!-\!7)(n\!-\!1)(n\!-\!2)(n\!-\!3))/24$ & $(n(n\!-\!2)(n\!-\!3)(4n^2\!-\!31n\!+\!51))/24$\\
& & $1$ & $((n\!-\!1)(n\!-\!2)(n\!-\!3)(5n\!-\!11))/6$ & $((n\!-\!3)(5n^3\!-\!46n^2\!+\!121n\!-\!86))/6$\\
& & $2$ & $((n\!-\!3)(20n\!-\!53)(n\!-\!2))/6$ & $(20n^3\!-\!213n^2\!+\!721n\!-\!774)/6$\\
& & $3$ & $(n\!-\!3)(10n\!-\!31)$ & $10n^2\!-\!81n\!+\!159$\\
& & $4$ & $20n\!-\!71$ & $20n\!-\!91$\\
$\{3,5,6\}$& $5$ & $0$ & $(n(n\!-\!1)(n\!-\!2)(n\!-\!3)(44n\!-\!161))/120$ & $(n(n\!-\!2)(n\!-\!3)(44n^2\!-\!425n\!+\!981))/120$\\
& & $1$ & $((n\!-\!1)(n\!-\!2)(n\!-\!3)(11n\!-\!41))/6$ & $((n\!-\!3)(11n^3\!-\!118n^2\!+\!367n\!-\!290))/6$\\
& & $2$ & $((n\!-\!3)(44n\!-\!167)(n\!-\!2))/6$ & $(44n^3\!-\!519n^2\!+\!1939n\!-\!2274)/6$\\
& & $3$ & $(n\!-\!3)(22n\!-\!85)$ & $22n^2\!-\!195n\!+\!417$\\
& & $4$ & $44n\!-\!173$ & $44n\!-\!217$\\
$\{4,5,6\}$& $5$ & $0$ & $(n(n\!-\!1)(n\!-\!2)(n\!-\!3)(22n\!-\!73))/60$ & $(n(n\!-\!2)(n\!-\!3)(22n^2\!-\!205n\!+\!453))/60$\\
& & $1$ & $((n\!-\!1)(n\!-\!2)(n\!-\!3)(11n\!-\!38))/6$ & $((n\!-\!3)(11n^3\!-\!115n^2\!+\!349n\!-\!272))/6$\\
& & $2$ & $((n\!-\!3)(22n\!-\!79)(n\!-\!2))/3$ & $(22n^3\!-\!255n^2\!+\!938n\!-\!1086)/3$\\
& & $3$ & $2(n\!-\!3)(11n\!-\!41)$ & $22n^2\!-\!192n\!+\!405$\\
& & $4$ & $2(22n\!-\!85)$ & $2(22n\!-\!107)$\\
$\{1,2,3,4\}$& $4$ & $0$ & $(n(n\!-\!1)(9n^2\!-\!37n\!+\!50))/24$ & $(n(9n^3\!-\!82n^2\!+\!243n\!-\!242))/24$\\
& & $1$ & $((n\!-\!1)(3n^2\!-\!13n\!+\!16))/2$ & $((n\!-\!3)(n\!-\!4)(3n\!-\!4))/2$\\
& & $2$ & $(9n^2\!-\!41n\!+\!48)/2$ & $((9n\!-\!23)(n\!-\!4))/2$\\
& & $3$ & $9n\!-\!25$ & $9n\!-\!34$\\
$\{1,2,3,5\}$& $5$ & $0$ & $(n(n\!-\!1)(4n^3\!-\!33n^2\!+\!97n\!-\!82))/24$ & $(n(4n^4\!-\!57n^3\!+\!298n^2\!-\!663n\!+\!514))/24$\\
& & $1$ & $((n\!-\!1)(5n^3\!-\!42n^2\!+\!121n\!-\!108))/6$ & $((n\!-\!3)(5n^3\!-\!52n^2\!+\!163n\!-\!128))/6$\\
& & $2$ & $(20n^3\!-\!171n^2\!+\!487n\!-\!444)/6$ & $(20n^3\!-\!231n^2\!+\!859n\!-\!1008)/6$\\
& & $3$ & $10n^2\!-\!67n\!+\!113$ & $(10n\!-\!37)(n\!-\!5)$\\
& & $4$ & $20n\!-\!77$ & $20n\!-\!97$\\
$\{1,2,4,5\}$& $5$ & $0$ & $(n(n\!-\!1)(2n^3\!-\!15n^2\!+\!41n\!-\!32))/12$ & $(n(n\!-\!4)(2n^3\!-\!19n^2\!+\!58n\!-\!53))/12$\\
& & $1$ & $((n\!-\!1)(5n^3\!-\!39n^2\!+\!106n\!-\!90))/6$ & $((n\!-\!3)(5n^3\!-\!49n^2\!+\!145n\!-\!110))/6$\\
& & $2$ & $(10n^3\!-\!81n^2\!+\!221n\!-\!195)/3$ & $(10n^3\!-\!111n^2\!+\!398n\!-\!453)/3$\\
& & $3$ & $2(5n^2\!-\!32n\!+\!52)$ & $10n^2\!-\!84n\!+\!173$\\
& & $4$ & $2(10n\!-\!37)$ & $2(10n\!-\!47)$\\
$\{1,2,5,6\}$& $5$ & $0$ & $(n(n\!-\!1)(11n^3\!-\!99n^2\!+\!296n\!-\!269))/30$ & $(n(11n^4\!-\!165n^3\!+\!890n^2\!-\!2025n\!+\!1619))/30$\\
& & $1$ & $((n\!-\!1)(11n^3\!-\!99n^2\!+\!292n\!-\!270))/6$ & $((n\!-\!3)(11n^3\!-\!121n^2\!+\!391n\!-\!314))/6$\\
& & $2$ & $(22n^3\!-\!198n^2\!+\!578n\!-\!537)/3$ & $(22n^3\!-\!264n^2\!+\!1007n\!-\!1203)/3$\\
& & $3$ & $2(11n^2\!-\!77n\!+\!133)$ & $22n^2\!-\!198n\!+\!431$\\
& & $4$ & $44(n\!-\!4)$ & $44(n\!-\!5)$\\
$\{1,3,4,5\}$& $5$ & $0$ & $(n(n\!-\!1)(4n^3\!-\!27n^2\!+\!59n\!-\!30))/24$ & $(n(n\!-\!5)(n\!-\!3)(4n^2\!-\!19n\!+\!18))/24$\\
& & $1$ & $((n\!-\!1)(5n^3\!-\!36n^2\!+\!85n\!-\!60))/6$ & $(5n^4\!-\!61n^3\!+\!259n^2\!-\!443n\!+\!246)/6$\\
& & $2$ & $(20n^3\!-\!153n^2\!+\!385n\!-\!312)/6$ & $(20n^3\!-\!213n^2\!+\!721n\!-\!768)/6$\\
& & $3$ & $(n\!-\!3)(10n\!-\!31)$ & $10n^2\!-\!81n\!+\!159$\\
& & $4$ & $20n\!-\!71$ & $20n\!-\!91$\\
$\{1,3,5,6\}$& $5$ & $0$ & $(n(n\!-\!1)(44n^3\!-\!381n^2\!+\!1069n\!-\!906))/120$ & $(n(n\!-\!3)(44n^3\!-\!513n^2\!+\!1831n\!-\!1902))/120$\\
& & $1$ & $((n\!-\!1)(11n^3\!-\!96n^2\!+\!271n\!-\!240))/6$ & $(11n^4\!-\!151n^3\!+\!721n^2\!-\!1385n\!+\!858)/6$\\
& & $2$ & $(44n^3\!-\!387n^2\!+\!1099n\!-\!996)/6$ & $((n\!-\!4)(44n^2\!-\!343n\!+\!567))/6$\\
& & $3$ & $(n\!-\!3)(22n\!-\!85)$ & $22n^2\!-\!195n\!+\!417$\\
& & $4$ & $44n\!-\!173$ & $44n\!-\!217$\\
$\{1,4,5,6\}$& $5$ & $0$ & $(n(n\!-\!1)(2n\!-\!3)(11n^2\!-\!75n\!+\!136))/60$ & $(n(n\!-\!3)(n\!-\!4)(22n^2\!-\!161n\!+\!219))/60$\\
& & $1$ & $((n\!-\!1)(11n^3\!-\!93n^2\!+\!256n\!-\!222))/6$ & $((n\!-\!4)(11n^3\!-\!104n^2\!+\!278n\!-\!201))/6$\\
& & $2$ & $(22n^3\!-\!189n^2\!+\!527n\!-\!471)/3$ & $(22n^3\!-\!255n^2\!+\!938n\!-\!1083)/3$\\
& & $3$ & $2(n\!-\!3)(11n\!-\!41)$ & $22n^2\!-\!192n\!+\!405$\\
& & $4$ & $2(22n\!-\!85)$ & $2(22n\!-\!107)$\\
$\{2,3,4,5\}$& $5$ & $0$ & $(n(n\!-\!1)(n\!-\!2)(4n^2\!-\!19n\!+\!29))/24$ & $(n(n\!-\!5)(n\!-\!2)(4n^2\!-\!23n\!+\!37))/24$\\
& & $1$ & $((n\!-\!1)(n\!-\!2)(5n^2\!-\!26n\!+\!39))/6$ & $(5n^4\!-\!61n^3\!+\!265n^2\!-\!479n\!+\!288)/6$\\
& & $2$ & $((n\!-\!2)(20n^2\!-\!113n\!+\!171))/6$ & $(20n^3\!-\!213n^2\!+\!733n\!-\!810)/6$\\
& & $3$ & $10n^2\!-\!61n\!+\!95$ & $(2n\!-\!7)(5n\!-\!23)$\\
& & $4$ & $20n\!-\!71$ & $20n\!-\!91$\\
$\{2,4,5,6\}$& $5$ & $0$ & $(n(n\!-\!1)(n\!-\!2)(22n^2\!-\!139n\!+\!239))/60$ & $(n(n\!-\!2)(22n^3\!-\!271n^2\!+\!1088n\!-\!1439))/60$\\
& & $1$ & $((n\!-\!1)(n\!-\!2)(11n^2\!-\!71n\!+\!120))/6$ & $(11n^4\!-\!148n^3\!+\!700n^2\!-\!1349n\!+\!846)/6$\\
& & $2$ & $((n\!-\!2)(22n^2\!-\!145n\!+\!243))/3$ & $((n\!-\!4)(22n^2\!-\!167n\!+\!276))/3$\\
& & $3$ & $2(11n^2\!-\!74n\!+\!124)$ & $22n^2\!-\!192n\!+\!407$\\
& & $4$ & $2(22n\!-\!85)$ & $2(22n\!-\!107)$\\
$\{3,4,5,6\}$& $5$ & $0$ & $(n(n\!-\!1)(n\!-\!2)(n\!-\!3)(44n\!-\!131))/120$ & $(n(n\!-\!2)(n\!-\!3)(44n^2\!-\!395n\!+\!831))/120$\\
& & $1$ & $((11n\!-\!35)(n\!-\!1)(n\!-\!2)(n\!-\!3))/6$ & $((n\!-\!3)(11n^3\!-\!112n^2\!+\!331n\!-\!254))/6$\\
& & $2$ & $((n\!-\!3)(44n\!-\!149)(n\!-\!2))/6$ & $(44n^3\!-\!501n^2\!+\!1813n\!-\!2070)/6$\\
& & $3$ & $(n\!-\!3)(22n\!-\!79)$ & $22n^2\!-\!189n\!+\!393$\\
& & $4$ & $44n\!-\!167$ & $44n\!-\!211$\\
$\{1,2,3,4,5\}$& $5$ & $0$ & $(n(n\!-\!1)(4n^3\!-\!27n^2\!+\!67n\!-\!46))/24$ & $(n(4n^4\!-\!51n^3\!+\!238n^2\!-\!477n\!+\!334))/24$\\
& & $1$ & $((n\!-\!1)(5n^3\!-\!36n^2\!+\!91n\!-\!72))/6$ & $((n\!-\!3)(n\!-\!4)(5n^2\!-\!26n\!+\!23))/6$\\
& & $2$ & $(20n^3\!-\!153n^2\!+\!397n\!-\!336)/6$ & $((n\!-\!4)(20n^2\!-\!133n\!+\!201))/6$\\
& & $3$ & $10n^2\!-\!61n\!+\!95$ & $(2n\!-\!7)(5n\!-\!23)$\\
& & $4$ & $20n\!-\!71$ & $20n\!-\!91$\\
$\{1,2,3,5,6\}$& $5$ & $0$ & $(n(n\!-\!1)(44n^3\!-\!381n^2\!+\!1109n\!-\!986))/120$ & $(n(44n^4\!-\!645n^3\!+\!3410n^2\!-\!7635n\!+\!6026))/120$\\
& & $1$ & $((n\!-\!1)(11n^3\!-\!96n^2\!+\!277n\!-\!252))/6$ & $((n\!-\!3)(11n^3\!-\!118n^2\!+\!373n\!-\!296))/6$\\
& & $2$ & $(44n^3\!-\!387n^2\!+\!1111n\!-\!1020)/6$ & $(44n^3\!-\!519n^2\!+\!1951n\!-\!2304)/6$\\
& & $3$ & $22n^2\!-\!151n\!+\!257$ & $22n^2\!-\!195n\!+\!419$\\
& & $4$ & $44n\!-\!173$ & $44n\!-\!217$\\
$\{1,2,4,5,6\}$& $5$ & $0$ & $(n(n\!-\!1)(22n^3\!-\!183n^2\!+\!517n\!-\!448))/60$ & $(n(n\!-\!4)(22n^3\!-\!227n^2\!+\!722n\!-\!697))/60$\\
& & $1$ & $((n\!-\!1)(11n^3\!-\!93n^2\!+\!262n\!-\!234))/6$ & $((n\!-\!3)(11n^3\!-\!115n^2\!+\!355n\!-\!278))/6$\\
& & $2$ & $(22n^3\!-\!189n^2\!+\!533n\!-\!483)/3$ & $(22n^3\!-\!255n^2\!+\!944n\!-\!1101)/3$\\
& & $3$ & $2(11n^2\!-\!74n\!+\!124)$ & $22n^2\!-\!192n\!+\!407$\\
& & $4$ & $2(22n\!-\!85)$ & $2(22n\!-\!107)$\\
$\{1,3,4,5,6\}$& $5$ & $0$ & $(n(n\!-\!1)(44n^3\!-\!351n^2\!+\!919n\!-\!726))/120$ & $(n(n\!-\!3)(44n^3\!-\!483n^2\!+\!1621n\!-\!1602))/120$\\
& & $1$ & $((n\!-\!1)(11n^3\!-\!90n^2\!+\!241n\!-\!204))/6$ & $(11n^4\!-\!145n^3\!+\!667n^2\!-\!1241n\!+\!750)/6$\\
& & $2$ & $(44n^3\!-\!369n^2\!+\!1009n\!-\!888)/6$ & $(44n^3\!-\!501n^2\!+\!1813n\!-\!2064)/6$\\
& & $3$ & $(n\!-\!3)(22n\!-\!79)$ & $22n^2\!-\!189n\!+\!393$\\
& & $4$ & $44n\!-\!167$ & $44n\!-\!211$\\
$\{2,3,4,5,6\}$& $5$ & $0$ & $(n(n\!-\!1)(n\!-\!2)(44n^2\!-\!263n\!+\!433))/120$ & $(n(n\!-\!2)(44n^3\!-\!527n^2\!+\!2056n\!-\!2653))/120$\\
& & $1$ & $((n\!-\!1)(n\!-\!2)(11n^2\!-\!68n\!+\!111))/6$ & $(11n^4\!-\!145n^3\!+\!673n^2\!-\!1277n\!+\!792)/6$\\
& & $2$ & $((n\!-\!2)(44n^2\!-\!281n\!+\!459))/6$ & $(44n^3\!-\!501n^2\!+\!1825n\!-\!2106)/6$\\
& & $3$ & $22n^2\!-\!145n\!+\!239$ & $(22n\!-\!79)(n\!-\!5)$\\
& & $4$ & $44n\!-\!167$ & $44n\!-\!211$\\
$\{1,2,3,4,5,6\}$& $5$ & $0$ & $(n(n\!-\!1)(44n^3\!-\!351n^2\!+\!959n\!-\!806))/120$ & $(n(44n^4\!-\!615n^3\!+\!3110n^2\!-\!6705n\!+\!5126))/120$\\
& & $1$ & $((n\!-\!1)(11n^3\!-\!90n^2\!+\!247n\!-\!216))/6$ & $((11n\!-\!13)(n\!-\!3)(n\!-\!4)(n\!-\!5))/6$\\
& & $2$ & $(44n^3\!-\!369n^2\!+\!1021n\!-\!912)/6$ & $((44n\!-\!105)(n\!-\!4)(n\!-\!5))/6$\\
& & $3$ & $22n^2\!-\!145n\!+\!239$ & $(22n\!-\!79)(n\!-\!5)$\\
& & $4$ & $44n\!-\!167$ & $44n\!-\!211$\\
\end{longtable}
\end{center}

 Note that the method   in Theorem \ref{symmetric-thm} is invalid for those $T=\cup_{i\in \mathcal{I}_T}\mathcal{C}^{(i)}$ with 
 \begin{equation}\label{bad}
 \mathcal{I}_T\in\left\{
 \begin{array}{l}
 \{1,3\},\{1,6\},\{4,6\},\{1,2,3\}, \{1,2,6\},\{1,3,6\},\{1,4,6\},\\
 \{2,4,6\},\{3,4,6\},\{1,2,3,6\},\{1,2,4,6\}, \{1,3,4,6\}, \\
 \{2,3,4,6\}, \{2,3,5,6\},\{1,2,3,4,6\}
 \end{array}
 \right\}.
 \end{equation}
 Thus we have the following problem:

\begin{prob}\label{prob-1}
For $T=\cup_{i\in \mathcal{I}_T}\mathcal{C}^{(i)}$ with $\mathcal{I}_T$  shown in  (\ref{bad}), what is the  second eigenvalue of the normal Cayley graph $G=\mathrm{Cay}(S_n,T)$?
\end{prob}

\begin{remark}\label{disconnected-rem}
It is worth mentioning that for small $m$ (for example, $m=6$ or $7$), as in Theorem \ref{symmetric-thm}, one can also  determine the second eigenvalues of some connected normal Cayley graphs (and some subgraphs of these graphs) of $S_n$ as long as the computer can verify the conditions of Theorem \ref{normal-thm}.
\end{remark}

\begin{remark}\label{symmetric-rem}
It is well known that the alternating group $A_n$ ($n\geq 3$) acts $(n-2)$-transitively on $[n]$. Thus the method used in Theoerm \ref{symmetric-thm} is still valid for determining the second eigenvalues of those connected normal Cayley graphs (and some subgraphs of these graphs) of $A_n$ when $m$ is  relatively small.
\end{remark}

Let $T=\mathcal{C}^{(1)}$ (see (\ref{eq-21})) be the set of all transpositions in $S_n$ ($n\geq 3$). Then $m=2$ and $T_1=T_{m-1}=\mathcal{C}_1^{(1)}=\{(1,q)\mid 2\leq q\leq n\}$. If $n\geq 7$, by Theorem \ref{symmetric-thm} (see also Table \ref{tab-2}), the spectral gap of $G=\mathrm{Cay}(S_n,T)$ and $G_1=\mathrm{Cay}(S_n,T_1)$ are  $|T|-|T\cap  (S_n) _1|+|T\cap  (S_n) _{2,1}|=\frac{1}{2}n(n-1)-\frac{1}{2}(n-1)(n-2)+1=n$ and $|T_1|-|T_1\cap  (S_n) _2|+|T_1\cap  (S_n) _{3,2}|=n-1-(n-2)+0=1$, respectively. If $3\leq n\leq 6$, one can easily verify that the result also holds. Thus, the two results below are consequences of our work.

\begin{cor}[Diaconis and Shahshahani \cite{Diaconis}]
For $n\geq 3$, the  spectral gap of $\mathrm{Cay}(S_n,\{(p,q)\mid 1\leq p,q\leq n\})$ is $n$.
\end{cor}

\begin{cor}[Flatto, Odlyzko and Wales \cite{Flatto}]
For $n\geq 3$, the  spectral gap of $\mathrm{Cay}(S_n,\{(1,q)\mid 2\leq q\leq n\})$ is $1$.
\end{cor}

\section{Further research}\label{s-5}
Let $\Gamma$ be finite group acts transitively on $[n]$ (for example, $\Gamma=S_n$ or $A_n$), and let $\mathrm{Cay}(\Gamma,T)$ be a Cayley graph of $\Gamma$. By Theorem \ref{Cay-thm}, the left coset decomposition  given in (\ref{eq-4})  is always an equitable partition of $\mathrm{Cay}(\Gamma,T)$, and the corresponding  quotient matrix $B_{\Pi}=(b_{s,t})_{n\times n}$ (see (\ref{eq-6})) is symmetric, where $b_{s,t}$ (=$b_{t,s}$) is the number of elements in $T$ moving $t$ to $s$. Since the eigenvalues of $B_{\Pi}$ are also eigenvalues of $\mathrm{Cay}(\Gamma,T)$, we have $\lambda_2(B_\Pi)\leq \lambda_2(\mathrm{Cay}(\Gamma,T))$. Inspired by the main result of Section \ref{s-4}, we pose the following problem.

\begin{prob}\label{prob-quotient}
Let $\Gamma$ be finite group acts transitively on $[n]$. For which connected  Cayley graphs of $\Gamma$, the equality $\lambda_2(B_\Pi)=\lambda_2(\mathrm{Cay}(\Gamma,T))$ holds?
\end{prob}

Let $T$ be a symmetric generating subset of $\Gamma$. We define the \emph{permutation graph} $\mathrm{Per}(T)$ as the edge-weighted graph with vertex set $\{1,2,\ldots,n\}$ in which each edge $e=st$ ($s\neq t$) has weight $w(e)=b_{s,t}$,  the  number of elements in $T$ moving $t$ to $s$ as mentioned above. If $\Gamma=S_n$ and $T$ contains only transpositions, it is clear that the permutation graph $\mathrm{Per}(T)$ coincides with the transposition graph $\mathrm{Tra}(T)$ defined in Section \ref{s-1}.  Since $\mathrm{Cay}(\Gamma,T)$ is $|T|$-regular,  the sum of each row of the quotient matrix $B_\Pi$ is equal to $|T|$. We can verify that $B_\Pi=|T|\cdot I_n-L(\mathrm{Per}(T))$, where  $L(\mathrm{Per}(T))$ is the Laplacian matrix of the permutation graph $\mathrm{Per}(T)$. This implies that $\lambda_2(B_\Pi)=|T|\cdot I_n-\mu_{n-1}(L(\mathrm{Per}(T)))$, where $\mu_{n-1}(L(\mathrm{Per}(T)))$ denotes the second least eigenvalue of $L(\mathrm{Per}(T))$, i.e., the algebraic connectivity of $\mathrm{Per}(T)$. Therefore, the spectral gap of $\mathrm{Cay}(\Gamma,T)$ satisfies the inequality 
$$|T|-\lambda_2(\mathrm{Cay}(\Gamma,T))\leq |T|-\lambda_2(B_\Pi)=\mu_{n-1}(L(\mathrm{Per}(T))).$$
Then we can  restate Problem \ref{prob-quotient} as below.

\begin{prob}\label{prob-spectral-gap}
Let $\Gamma$ be finite group acts transitively on $[n]$. For which connected  Cayley graphs of $\Gamma$, the spectral gap of $\mathrm{Cay}(\Gamma,T)$ equals to the algebraic connectivity of the permutation graph $\mathrm{Per}(T)$?
\end{prob}

In fact,  Aldous' theorem give a positive answer of Problem  \ref{prob-quotient} (or Problem \ref{prob-spectral-gap})  in the case that $\Gamma=S_n$ and $T$ consists of transpositions. Also,  the result of Theorem \ref{symmetric-thm} in this paper gives a  partial answer of Problem  \ref{prob-quotient} (or Problem \ref{prob-spectral-gap}) for  the connected normal Cayley graphs (and some of their subgraphs) of $S_n$  with  $\max_{\tau\in T}|\mathrm{supp}(\tau)|\leq 5$.

For any $\sigma\in S_n$,  there exists a unique partition $[n] = I_1\cup \cdots \cup I_m$ of $[n]$ into contiguous blocks such that $\sigma(I_i) = I_i$ for each $i\in [m]$. Here, each $I_i$ consists of consecutive elements in $[n]$, so that $I_i=\{a,a + 1,\ldots,b\}$ for some pair of natural numbers $a \leq b$. If this partition is of cardinality $m$, then we call $\sigma$ an $m$-\emph{reducible permutation}. 
In \cite{Dai1,Dai2}, Dai introduced and discussed some combinatorial properties of a new variant of the family of Johnson graphs, the Full-Flag Johnson graphs. He showed that the Full-Flag Johnson graph $FJ(n,r)$ ($r<n$) is isomorphic to the Cayley graph $\mathrm{Cay}(S_n,RP^{(r)})$, where $RP^{(r)}$ is the set of all $(n-r)$-reducible permutations of $S_n$. For a positive integer $n$,  the Cayley graph $\mathrm{Cay}(S_n,\{(i,i+1)\mid 1\leq i\leq n-1\})$ is called the \emph{permutahedron} of order $n$, which is a well-known combinatorial graph. Observe that each $(n-1)$-reducible permutation of $S_n$ must be of the form $(i,i+1)$ for some $i\in[n-1]$, we have $RP^{(1)}=\{(i,i+1)\mid 1\leq i\leq n-1\}$, and so the permutahedron of order $n$ is just the Full-Flag Johnson graph $FJ(n,1)$. Thus the Full-Flag Johnson graphs can be also viewed as the generalizations of permutahedra \cite{Dai2}.

Let  $M_n$ be  the tridiagonal matrix of order $n$ defined as below:
$$
M_n=\left[\begin{smallmatrix}
n-2 & 1& 0& 0&\cdots &0& 0& 0& 0\\
1& n-3&1&0&\cdots&0&0& 0& 0\\
0 &1 &n-3 &1& \cdots &0& 0& 0& 0\\
&&&&\vdots&&&&\\
0& 0& 0& 0&\cdots&1&n-3& 1& 0&\\
0& 0& 0&0&\cdots &0&1&n-3&1\\
0& 0& 0&0&\cdots&0&0 &1&n-2
\end{smallmatrix}
\right].
$$
At the end of the paper \cite{Dai2}, Dai  proved that the eigenvalues of $M_n$ are also eigenvalues of the permutahedron $FJ(n,1)$, and conjectured that $\lambda_2(M_n)=\lambda_2(FJ(n,1))$. In fact, since $FJ(n,1)=\mathrm{Cay}(S_n,RP^{(1)})$ with $RP^{(1)}=\{(i,i+1)\mid 1\leq i\leq n-1\}$,  $M_n$ is just  the quotient matrix of $FJ(n,1)$ shown in (\ref{eq-6}). Thus we may conclude that  Dai's conjecture follows from Aldous' theorem immediately by the arguments at the beginning of this section. 

Now consider the graph $FJ(n,2)=\mathrm{Cay}(S_n,RP^{(2)})$ where $RP^{(2)}$ consists of all $(n-2)$-reducible permutations of $S_n$. By definition, we can check that each $(n-2)$-reducible permutation of $S_n$ belongs to one of the following three classes:
\begin{equation*}
\left\{
\begin{aligned}
Q^{(1)}&=\{(i,i+1,i+2),(i,i+2,i+1)\mid 1\leq i\leq n-2\};\\
Q^{(2)}&=\{(i,i+2)\mid 1\leq i\leq n-2\};\\
Q^{(3)}&=\{(i,i+1)(j,j+1)\mid 1\leq i\leq n-3,3\leq j\leq n-1,i< j-1\}.
\end{aligned}
\right.
\end{equation*}
Therefore, we have $RP^{(2)}=Q^{(1)}\cup Q^{(2)}\cup Q^{(3)}$. Furthermore, by Theorem \ref{Cay-thm} and (\ref{eq-6}), the graph $FJ(n,2)=\mathrm{Cay}(S_n,RP^{(2)})$ has the quotient matrix
$$
B_n=\left[
\begin{smallmatrix}
\frac{n^2-n-6}{2} & n-2 & 2 & 0 &  0 &0 & \cdots & 0 & 0 & 0 &0 &0 & 0 \\
n-2 & \frac{n^2-3n-2}{2} & n-2 & 2 & 0 &0 &  \cdots & 0&0 &0 &0 & 0&0 \\
2 & n-2 & \frac{n^2-3n-6}{2}& n-2 & 2 &0 &  \cdots& 0&0 &0 &0 &0 &0 \\
0 & 2 & n-2 & \frac{n^2-3n-6}{2} & n-2 & 2 &  \cdots &0 &0 &0 &0 &0 &0 \\
&&&&&& \vdots&&&&&&\\
0&0&0&0&0&0& \cdots&2 & n-2 & \frac{n^2-3n-6}{2} & n-2 & 2&0\\
0&0&0&0&0&0& \cdots&0&2 & n-2 & \frac{n^2-3n-6}{2} & n-2 & 2\\
0&0&0&0&0&0& \cdots&0&0&2 & n-2 & \frac{n^2-3n-2}{2} & n-2\\
0&0&0&0&0&0& \cdots&0&0&0&2 & n-2 & \frac{n^2-n-6}{2}\\
\end{smallmatrix}
\right]_{n\times n}.
$$
In accordance with Problem \ref{prob-quotient},  we ask if $\lambda_2(FJ(n,2))=\lambda_2(B_n)$?  Using computer, we  can verify that the equality holds for $4\leq n\leq 7$ and we make the following conjecture.
\begin{conj}\label{conj-1}
For $n\geq 4$,  $\lambda_2(FJ(n,2))=\lambda_2(B_n)$.
\end{conj}
Theorem \ref{Cay-thm} indicates a possible method to prove Conjecture \ref{conj-1}. Now we describe the detail of the method. For $k=1,2$, we define $$FJ_k(n,2)=\mathrm{Cay}(S_n,RP_k^{(2)}),$$ where $RP_1^{(2)}=\{(1,2,3),(1,3,2),(1,3),$ $(1,2)(3,4),(1,2)(4,5),\ldots,(1,2)(n-1,n)\}$ and $RP_2^{(2)}=\{(1,2)(n-1,n)\}$. Note that $RP_1^{(2)}$ is the set of elements in $RP^{(2)}=Q^{(1)}\cup Q^{(2)}\cup Q^{(3)}$ moving $1$ while $RP_2^{(2)}$ is the set of elements in $RP_1^{(2)}$ moving $n$. Clearly, $FJ_1(n,2)$ is connected and   $FJ_2(n,2)$ is just the disjoint union of $\frac{n!}{2}$ $K_2$'s. Again by Theorem \ref{Cay-thm}, the graph $FJ_1(n,2)$ has the quotient matrix 
$$
B_n^{(1)}=\left[
\begin{smallmatrix}
0 & n-2 & 2 & 0 &  0 &\cdots & 0 & 0 & 0 &0 \\
n-2 & 1 & 1 & 0 & 0 &\cdots  & 0&0 &0 &0  \\
2 & 1 & n-4 & 1 & 0 &\cdots & 0&0 &0 &0  \\
0 & 0 & 1 & n-2 & 1 & \cdots  &0 &0 &0 &0  \\
&&&&& \vdots&&&&\\
0&0&0&0&0& \cdots&1 & n-2 & 1 & 0 \\
0&0&0&0&0& \cdots&0&1 & n-2 & 1 \\
0&0&0&0&0& \cdots&0&0&1 & n-1 \\
\end{smallmatrix}
\right]_{n\times n}.
$$
Using computer, we can check that $\lambda_2(FJ_1(n,2))=\lambda_2(B_n^{(1)})$ holds for $4\leq n\leq 7$, and so we propose the following conjecture.
\begin{conj}\label{conj-2}
For $n\geq 4$,  $\lambda_2(FJ_1(n,2))=\lambda_2(B_n^{(1)})$.
\end{conj}
In order to prove Conjecture \ref{conj-1} by induction on $n$, we can assume that the result holds for $n-1$, i.e., $\lambda_2(FJ(n-1,2))=\lambda_2(B_{n-1})$.   By the arguments below Theorem \ref{Cay-thm} and (\ref{eq-8}), it suffices to show that
$$
\lambda_2(B_n)\geq \lambda_2(\mathrm{Cay}((S_n)_1,RP^{(2)}\cap (S_n)_1))+\lambda_2(\mathrm{Cay}(S_n,RP^{(2)}\setminus (RP^{(2)}\cap (S_n)_1))).
$$
Note that $\mathrm{Cay}((S_n)_1,RP^{(2)}\cap (S_n)_1)\cong FJ(n-1,2)$ and $\mathrm{Cay}(S_n,RP^{(2)}\setminus (RP^{(2)}\cap (S_n)_1))=\mathrm{Cay}(S_n,RP^{(2)}_1)=FJ_1(n,2)$. Thus, if Conjecture \ref{conj-2} is true, it remains to verify the following inequality:
\begin{equation}\label{eq-26}
\lambda_2(B_n)\geq \lambda_2(B_{n-1})+\lambda_2(B_n^{(1)}).
\end{equation}
Thus  we also need to  prove Conjecture \ref{conj-2}. As above, we can  assume $\lambda_2(FJ_1(n-1,2))=\lambda_2(B_{n-1}^{(1)})$, and  it suffices to show that 
\begin{equation}\label{eq-27}
\begin{aligned}
\lambda_2(B_n^{(1)})&\geq \lambda_2(\mathrm{Cay}((S_n)_n,RP^{(2)}_1\cap (S_n)_n))+\lambda_2(\mathrm{Cay}(S_n,RP^{(2)}_1\setminus (RP^{(2)}_1\cap (S_n)_n)))\\
&=\lambda_2(FJ_1(n-1,2))+\lambda_2(FJ_2(n,2))\\
&=\lambda_2(B_{n-1}^{(1)})+1,\\
\end{aligned}
\end{equation}
here we use the facts $\mathrm{Cay}((S_n)_n,RP^{(2)}_1\cap (S_n)_n)\cong FJ_1(n-1,2)$ and $\mathrm{Cay}(S_n,RP^{(2)}_1\setminus (RP^{(2)}_1\cap (S_n)_n))=FJ_2(n,2)\cong \frac{n!}{2}K_2$. Therefore, if one can prove (\ref{eq-26}) and  (\ref{eq-27}), then Conjecture \ref{conj-1} and Conjecture \ref{conj-2} follows immediately. However, it is not easy to identify the second eigenvalues of  $B_n$ and $B_n^{(1)}$, so we leave it as an open problem.

In accordance with Problem \ref{prob-quotient},  for  $r\geq 3$,   we pose the following  problem.
\begin{prob}\label{prob-2}
For $3\leq r<n$, does the quotient matrix  given in  (\ref{eq-6}) always contain the second eigenvalue of  the Full-Flag graph $FJ(n,r)=\mathrm{Cay}(S_n,RP^{(r)})$?
\end{prob}

On the other hand, for regular graphs, the smallest eigenvalue  is closely related to the independent number. Let $G$ be a $k$-regular graph $G$ with smallest eigenvalue $\tau$ and independent number $\alpha(G)$, the well-known Hoffman ratio bound asserts that  
$$
\alpha(G)\leq \frac{|V(G)|}{1-k/\tau},
$$
and that if the equality holds for some independent set $S$ with characteristic vector $v_S$, then $v_S-\frac{|S|}{|V(G)|}\mathbf{1}$ is an eigenvector of the eigenvalue $\tau$. By applying the Hoffman ratio bound to several important families of graphs belonging to classical $P$- or $Q$-polynomial association schemes (such as Johnson scheme, Hamming
scheme, Grassmann scheme) and  some famous Cayley graphs (such as the derangement graph) on the symmetric group $S_n$, variants of  Erd\H{o}s-Ko-Rado Theorems for sets, vector spaces, integer sequences and permutations have been obtained by various researchers (see  Godsil and Meagher \cite{Godsil2} for the detail).  Recently, Brouwer, Cioab\u{a}, Ihringer and McGinnis \cite{Brouwer1} determine the smallest eigenvalues of (distance-$j$) Hamming graphs, (distance-$j$) Johnson
graphs, and  the graphs of the relations of classical $P$- and $Q$-polynomial association schemes. Motivated by these works, it is interesting to consider the smallest eigenvalues of  normal Cayley graphs of  $S_n$. A natural question is that whether the method developed in this paper is valid for the smallest eigenvalues. However, it is not the case. According to the proof of Lemma \ref{quotient-lem2}, the quotient matrix $B_\Pi$ ($=B_\Pi^{0}$) of the normal Cayley graph $G_0=\mathrm{Cay}(S_n,T_0=T)$  has eigenvalue $|T|$ and $|T\cap \Gamma_{1}|-|T\cap \Gamma_{2,1}|$ (with multiplicity $n-1$). Thus we have $\lambda_n(B_\Pi)=\lambda_2(B_\Pi)=|T\cap \Gamma_{1}|-|T\cap \Gamma_{2,1}|$. If $n\geq 7$, we can verify that 
$\lambda_n(B_\Pi)=\lambda_2(B_\Pi)\geq 0$ holds for all connected normal Cayley graphs of $S_n$ with $\max_{\tau\in T}|\mathrm{supp}(\tau)|\leq 5$, which implies that $\lambda_n(B_\Pi)$ cannot be the smallest eigenvalue. Thus we  pose the following problem.
\begin{prob}
For normal Cayley graphs of $S_n$, are there some good general methods to  determine the smallest eigenvalues? 
\end{prob}

\end{document}